\documentclass{amsart}

\usepackage[utf8]{inputenc}
\usepackage[mathscr]{eucal}

\usepackage{amssymb}
\usepackage{amsmath}
\usepackage{amsthm}
\usepackage{enumitem}

\usepackage{pgf,tikz}
\usetikzlibrary{arrows,matrix}

\numberwithin{equation}{section}
\theoremstyle{plain}
\newtheorem{thm}{Theorem}[section]
\newtheorem{lemma}[thm]{Lemma}
\newtheorem{prop}[thm]{Proposition}
\newtheorem{cor}[thm]{Corollary}
\newtheorem{claim}[thm]{Claim}

\newtheorem{problem}[thm]{Problem}

\theoremstyle{definition}
\newtheorem{defn}[thm]{Definition}
\newtheorem{ex}[thm]{Example}

\newtheorem{conject}[thm]{Conjecture}

\newtheorem{remark}[thm]{Remark}

\usepackage{xspace} 

\newtheorem{innercustomthm}{Theorem}
\newenvironment{customthm}[1]
{\renewcommand\theinnercustomthm{#1}\innercustomthm}
{\endinnercustomthm}

\newcommand{\R}{\mathbf{R}}
\newcommand{\Z}{\mathbf{Z}}

\newcommand{\ehr}{\mathrm{Ehr}}

\DeclareMathOperator{\conv}{Conv}
\DeclareMathOperator{\dist}{dist}

\def\jaeg{{\mathcal{J}}}
\def\cG{{\mathcal{G}}}

\definecolor{light-gray}{gray}{0.8}
\definecolor{v}{rgb}{0.28,0,0.72}
\definecolor{e}{rgb}{0,1,0.2}
\definecolor{r}{rgb}{1,0,0}

\begin{document}

\title[Ehrhart theory of symmetric edge polytopes via ribbon structures]{Ehrhart theory of symmetric edge polytopes via ribbon structures}

\author{Tam\'as K\'alm\'an}
\address{Department of Mathematics\\
Tokyo Institute of Technology\\
H-214, 2-12-1 Ookayama, Meguro-ku, Tokyo 152-8551, Japan}
\email{kalman@math.titech.ac.jp}

\author{Lilla T\'othm\'er\'esz}
\address{MTA-ELTE Egerv\'ary Research Group, P\'azm\'any P\'eter s\'et\'any 1/C, Budapest, Hungary}
\email{tmlilla@caesar.elte.hu}
\date{}

\begin{abstract}
Using a ribbon structure of the graph, we construct a dissection of the symmetric edge polytope of a graph into unimodular simplices. 
Our dissection is shellable, and one can interpret the elements of the resulting $h$-vector via graph theory. This gives an elementary method for computing the $h^*$-vector of the symmetric edge polytope.

\end{abstract}

\maketitle

\section{Introduction}
\label{sec:intro}

Let $\cG$ be a simple graph with vertex set $V(\cG)$ and edge set $E(\cG)$.
The polyhedron $$P_\cG =\conv \{\mathbf{u}-\mathbf{v}, \mathbf{v}-\mathbf{u} \mid uv\in E(\cG)\} \subset \mathbb{R}^{V(\mathcal{G})}$$ is called the \emph{symmetric edge polytope} of $\cG$ \cite{Sym_edge_appearance}.
(Here $\mathbf u, \mathbf v$ stand for generators of $\mathbb R^{V(\cG)}$ that correspond to $u,v\in V(\cG)$.)
The $h^*$-vector of the symmetric edge polytope is a palindromic polynomial that has 
been actively investigated; 
see \cite{arithm_symedgepoly,DDM19,Sym_edge_appearance,OT} and references within. 
In particular, Ohsugi and Tsuchiya conjecture \cite{OT} that the $h^*$-vector of a symmetric edge polytope is $\gamma$-positive. 
As pointed out in \cite{DDM19}, the symmetric edge polytope is also relevant in physics, where an upper bound for the number of steady states of the Kuramoto synchronization model can be obtained via its volume. 

In this paper we give a general method for computing the $h^*$-vector of the symmetric edge polytope of a graph. Our method only uses elementary notions of graph theory and can be converted into an exponential time algorithm.

Higashitani, Jochemko, and Micha{\l}ek described the facets of $P_\cG$ \cite{arithm_symedgepoly} and in a previous paper \cite{semibalanced} we treated exactly that class of polytopes. 
The facets are indexed by certain spanning subgraphs of $\cG$ that are also endowed with an orientation.
Here we build on the machinery of \cite{semibalanced} 
to dissect the symmetric edge polytope into unimodular simplices in a shellable manner. 
The basic idea is to dissect each facet in a shellable way as we did in \cite{semibalanced}, construct cones over them whose common apex is the origin, and merge the shellings into a shelling order of the whole construction.
This way we can read off the $h^*$-polynomial as the $h$-polynomial of the shelling.

We interpret the simplices in the dissection, as well as the coefficients of the $h^*$-polynomial, in terms of graph theory. More exactly, the simplices in the dissection correspond to special spanning trees (so called Jaeger trees) of the aforementioned oriented spanning subgraphs of $\cG$. Here Jaeger trees, see Definition \ref{def:jaegertree}, are defined using an arbitrarily fixed ribbon structure of $\cG$ and a graph traversal due to Bernardi \cite{Bernardi_first}. 
Roughly speaking, the condition is that each edge that is \emph{not} part of the tree, is first reached during the traversal at its tail.
The notion of these trees is inspired by the knot-theoretical work of F. Jaeger \cite{Jaeger}. 
Our main claim is that 

\begin{customthm}{\ref{cor:h^*_of_sym_edge_poly}}
The coefficient $(h^*_{P_\cG})_i$ equals the number of Jaeger trees so that, during the traversal, there are exactly $i$ edges that are \emph{in the tree} and are first reached at their tail.
\end{customthm}

The resulting computation of $h^*_{P_\cG}$ is lengthy (due to the large number of oriented subgraphs to be considered and the multitude of Jaeger trees in each) but eminently doable for any graph. In particular, no polytopes are considered during the  actual process.

Furthermore, one can use our formula to prove that the coefficient $\gamma_1$ of the linear term of the $\gamma$-polynomial of $h^*_{P_\cG}$ is nonnegative for any graph $\cG$. More precisely, we show 

\begin{customthm}{\ref{thm:gamma(1)=2g}}
    For any connected, simple, undirected graph $\cG$ we have $$\gamma_1=2g,$$ where $g=|E(\cG)|-|V(\cG)|+1$ is the so called cyclomatic number (or nullity or first Betti number) of $\cG$.
\end{customthm}

The same formula appears in an extremely recent announcement\footnote{Their preprint \cite{dalietal} was submitted one day ahead of ours.} by D'Al\`i et al. They show this by a quick calculation based on the existence of a unimodular triangulation, cf.\ \cite[Lemma 3.1]{dalietal}. Our approach yields a much longer proof but it does provide an explicit description of the simplices in the shelling order that are attached along exactly one facet.

If $\cG$ is bipartite with partite classes $U$ and $W$, the \emph{root polytope} 
\[\mathcal{Q}_\cG=\conv\{\mathbf{u}-\mathbf{w}\mid u\in U, w\in W, uw\in E(\cG)\}\]
is a facet of the symmetric edge polytope. 
We will also call the $h^*$-vector of $\mathcal{Q}_\cG$ the \emph{interior polynomial} of $\cG$.
(Cf.\ \cite{KP_Ehrhart} and Kato's clarification \cite{Kato} of its main result. The interior polynomial was first defined in \cite{hiperTutte}.)
Ohsugi and Tsuchiya proved \cite[Theorem 5.3]{OT} that for a certain special class of graphs, the $\gamma$-polynomial of $h^*_{P_\cG}$ is a positive linear combination of interior polynomials of some subgraphs. Such a formula implies $\gamma$-positivity for the symmmetric edge polytope of these graphs. In Section \ref{sec:connection_of_gamma_and_interior}, we collect further results suggesting a strong connection of the $\gamma$ and interior polynomials, and formulate some conjectures.

Previous results about the $h^*$-vector of the symmetric edge polytope mostly relied on Gröbner bases of the toric ideal associated to $P_\cG$, and a resulting regular unimodular triangulation of $P_\cG$.
Even though our approach and the method of Gröbner bases are conceptually very different, they sometimes produce the same subdivision. For example in the case of a complete bipartite graph, with the right ribbon structure, our method yields the same triangulation that was  obtained using Gröbner bases \cite{arithm_symedgepoly} (see Section \ref{sec:concrete_cases}).
To summarize the comparison, our approach is less algebraic and more combinatorial. Gröbner bases provide regular triangulations for the symmetric edge polytope, while we can only guarantee a shellable dissection. 
To us, it is not clear whether Gröbner bases directly yield a graph-theoretic interpretation for the $h^*$-vector the way our method does.

To demonstrate the usefulness of our method in a more concrete setting, in Section \ref{sec:concrete_cases}, we revisit some further graph classes where $h^*_{P_\cG}$ was computed previously using Gröbner bases. 


Finally, in Section \ref{sec:geometric_formula_for_volume}, we give a geometric formula for the volume of the symmetric edge polytope of a bipartite graph. Namely, we identify a set of points such that the volume of the polytope is equal to the sum of the numbers of facets visible from the given points.

Acknowledgements:
TK was supported by a Japan Society for the Promotion of Science (JSPS) Grant-in-Aid for Scientific Research C (no.\ 17K05244). LT was supported by the National Research, Development and Innovation Office of Hungary -- NKFIH, grant no.\ 132488, by the János Bolyai Research Scholarship of the Hungarian Academy of Sciences, and by the ÚNKP-21-5 New National Excellence Program of the Ministry for Innovation and Technology, Hungary. 
LT was also partially supported by the Counting in Sparse Graphs Lendület Research Group of the Rényi Institute.

\section{Preparations} \label{sec:prep}

\subsection{Ehrhart theory essentials}

Suppose that $P\subset\R^n$ is a $d$-dimensional
polytope with vertices in $\Z^n$. 
We define its \emph{Ehrhart series} as $$\ehr_P(t)=\sum_{k=0}^\infty|(k\cdot P)\cap\Z^n|\,t^k.
$$
It is well known that the Ehrhart series can be written in the closed form
$$
\ehr_P(t) = \frac{h^*(t)}{(1-t)^{d+1}},
$$
where $h^*(t)$ is a polynomial
called the \emph{$h^*$-polynomial} (or \emph{$h^*$-vector}) of $P$.

One possible way to compute the $h^*$-vector of $P$ is via the $h$-vector of a shellable dissection of $P$ into unimodular simplices, that we now explain.
This technique was already established for triangulations before we extended its scope in our earlier work \cite{hyperBernardi,semibalanced}, but we summarize it here again for easier readability.

A \emph{dissection} of a polytope $P$ is a set of mutually interior-disjoint maximal dimensional simplices, whose union is the polytope.
A dissection is \emph{shellable} if it admits a \emph{shelling order}. That means that the simplices of the dissection are listed as $\sigma_1,\sigma_2,\ldots,\sigma_N$, in such a way that for each $i=2,3,\ldots,N$, the intersection of $\sigma_i$ with the `earlier' simplices,
\[ \sigma_i\cap\left(\bigcup_{j=1}^{i-1}\sigma_j\right), \]
coincides with the union of a positive number of facets (codimension $1$ faces) of $\sigma_i$. 

Let us denote the number of said facets by $r_i$ for $i=2,3,\ldots,N$ and let us also put $r_1=0$. Then the \emph{$h$-vector} associated to the shelling order is the distribution of the statistic $r_i$. 
We may write it as a finite sequence $(h_0,h_1,\ldots)$, where $h_k$ is the number of simplices $\sigma_i$ with $r_i=k$. Note that $h_0=1$ and since each maximal simplex has $d+1$ facets, the subscript of the last non-zero term is at most $d+1$. In fact, it is at most $d$ by a topological argument: each $\sigma_i$ with $r_i=d+1$ (i.e., a $d$-cell attached along its entire boundary)
would add an infinite cyclic summand to the $d$-dimensional homology group of the polytope, but that group is $0$.

Alternatively, we may express the $h$-vector in the polynomial form
\begin{equation}
\label{eq:h-hejazott}
h(x)=h_{d}\,x+h_{d-1}\,x^2+\cdots+h_1\,x^{d}+h_0\,x^{d+1}.
\end{equation}

Now in the special case when the simplices in the dissection are unimodular, we have the following result, which will be a basic tool in this paper.

\begin{prop}
\label{prop:h-dissect}\cite{semibalanced}
For any shellable dissection of a $d$-dimensional lattice polytope into unimodular simplices, and for any shelling order, the $h$-polynomial \eqref{eq:h-hejazott} determines the $h^*$-polynomial of the polytope as
\[h^*(t)=t^{d+1}h(1/t)=h_0+h_1\,t+\cdots+h_{d-1}\,t^{d-1}+h_d\,t^d.\]
\end{prop}

If a polynomial $p(x)=c_dx^d+\dots + c_1x + c_0$ is palindromic (that is, $c_i=c_{d-i}$ for each $i=1, \dots, d$), then it can be written in the form 
$$p(x)= (x+1)^d \sum_{i=1}^{\lceil \frac{d}{2}\rceil} \gamma_i \cdot \left(\frac{x}{(x+1)^2}\right)^i.$$ 
In this case $\gamma(y) = \sum_{i=1}^{\lceil \frac{d}{2}\rceil} \gamma_i\, y^i$ is called the \emph{$\gamma$-polynomial} of $p$.

A palindromic polynomial $p$ is called \emph{$\gamma$-positive} if all coefficients of its $\gamma$-polynomial are nonnegative. This is a strong property, which implies the unimodality of $p$.

The $h^*$-polynomial of the symmetric edge polytope is palindromic \cite{OT}, hence $P_\cG$ has a $\gamma$-polynomial
for any graph $\cG$. We denote this polynomial by $\gamma_\cG$ and call it the \emph{$\gamma$-polynomial} of $\cG$.

\subsection{Graphs and digraphs}

In this paper 
we denote undirected graphs by calligraphic letters, e.g. $\cG$, while we denote directed graphs by regular capital letters, e.g., by $G$. 

For an undirected graph $\cG$, we denote its set of vertices by $V(\cG)$ and its set of edges  by $E(\cG)$. We write undirected edges as $uv$, where $u,v\in V(\cG)$ are the two endpoints.

For a directed graph $G$, we also denote its set of vertices by $V(G)$ and its set of (directed) edges by $E(G)$. 
We denote an edge of $G$ by $\overrightarrow{uv}$ if $u$ is the tail and $v$ is the head of the edge. If we do not want to specify which endpoint of the edge is the head and which one is the tail, we simply write $uv$. A lowercase letter, say $e$, might denote a directed or an undirected edge, depending on the context. We write $\overrightarrow{e}$ if we want to stress that $e$ is directed. Occasionally, $\overrightarrow{e}$ and $\overleftarrow{e}$ will denote the two oppositely oriented versions of the edge $e$.

In this paper, we typically consider directed graphs $G$ whose vertex sets agree with the vertex set of some undirected graph $\cG$, and the edge set of $G$, if we ignore the orientations, is a subset of the edge set of $\cG$. 
If we want to emphasize that $uv\in E(\cG)$ is also in $E(G)$, then we say that $uv$ is \emph{present} in $G$. 
If $G$ is an oriented subgraph of $\cG$ and $uv\in E(\cG)-E(G)$, then we say that $uv$ is a \emph{hidden edge} for $G$. 
In our figures, we will denote hidden edges by dotted lines. See Figure \ref{fig:facet_graphs} for an undirected graph and two oriented subgraphs of it. 

A \emph{cut} in a graph $\cG=(V,E)$ is a set of edges $C^*$ that contains exactly the edges going between $V_0$ and $V_1$ for some partition (into nonempty parts) $V_0 \sqcup V_1 = V$. We sometimes denote a cut by its vertex partition $(V_0,V_1)$, and call $V_0$ and $V_1$ the \emph{sides} of the cut.
In a directed graph,
we say that $C^*$ (or $(V_0,V_1)$) is a \emph{directed cut} if each edge points in the same direction, i.e., all of them point from $V_0$ to $V_1$ or all of them point from $V_1$ to $V_0$. Otherwise we say that the cut is \emph{undirected}.
We remark that in connected graphs, the edges of the cut do uniquely determine the corresponding vertex partition. Working with vertex partitions has the advantage that we can specify cuts in various subgraphs of $\cG$ at once, even though the cuts will typically be different when viewed as sets of edges. 

For a digraph $G$ and a set of vertices $V_1 \subset V$, we denote by $G[V_1]$ the graph whose vertex set is $V_1$, and whose edge set is the set of edges of $G$ where both endpoints are from $V_1$.

For a connected graph 
or weakly connected digraph, 
a \emph{spanning subgraph} is a subgraph that contains all vertices 
and is itself connected/weakly connected.

A \emph{spanning tree} in a directed or undirected graph is a cycle-free spanning subgraph. (Thus, in a directed graph, a subgraph is a spanning tree if and only if it is a spanning tree when we ignore the orientations.)
If $T$ is a spanning tree of a graph $\cG$, then for any edge $e\in T$, the subgraph $T-e$ has exactly two connected components. Let $V_0$ and $V_1$ be the vertex sets of these components and call the cut $(V_0,V_1)$ of $\cG$, denoted by $C^*_\cG(T,e)$, the \emph{fundamental cut} of $e$ with respect to $T$. (We may omit the subscript if the graph is clear from the context.) We define fundamental cuts for spanning trees of directed graphs the same way (i.e., edge directions do not play a role in the definiton).

\subsection{Symmetric edge polytopes and their facets}
Let $\cG$ be a graph with vertex set $V$. For $v\in V$, we denote by $\mathbf{v}$ the vector in $\mathbf{R}^V$ whose coordinate corresponding to $v$ is 1, and all the other coordinates are 0.

\begin{defn} The \emph{symmetric edge polytope} of $\cG$ is
    $$P_\cG =\conv\{\mathbf{u}-\mathbf{v}, \mathbf{v}-\mathbf{u} \mid uv\in E(\cG)\} \subset \mathbb{R}^{V(\mathcal{G})}.$$
\end{defn}

We note that $P_\cG$ lies along the hyperplane of $\mathbb{R}^{V(\mathcal{G})}$ where the sum of components is $0$.
In \cite{arithm_symedgepoly}, the facets of the symmetric edge polytope are characterized as follows. 

\begin{thm}\cite[Theorem 3.1]{arithm_symedgepoly}\label{thm:facets_of_P_G}
	For a graph $\cG$, the facets of $P_\cG$ are enumerated by layerings $l\colon V\to \mathbb{Z}$ such that 
	\begin{itemize}
		\item[(i)] $|l(v)-l(u)|\leq 1$ for each $uv\in E(\cG)$
		\item[(ii)] The subset of edges $E_l=\{
		{uv}\mid uv\in E(\cG), |l(v)-l(u)|=1\}$ forms a spanning subgraph of $\cG$. (In cases when $\cG$ is disconnected, this means that $(V,E_l)$ has the same number of 
		connected components as $\cG$.)
	\end{itemize}
	For such an $l$, the corresponding facet is $\conv\{\mathbf v - \mathbf{u}\mid 
	{uv}\in E_l\text{ and }l(v)-l(u)=1\}$.
\end{thm}


For each $l$ as above, the edges in $E_l$ are naturally directed from $u$ to $v$ where $l(v)-l(u)=1$.
Let us call the directed graphs $(V,E_l)$, arising as in the theorem, the \emph{facet graphs} of $\cG$. (They are subgraphs of $\cG$ that are also given an orientation.) For an example, see Figure \ref{fig:facet_graphs}.
Notice that (as we suppose that $\cG$ is connected) two layerings determine the same facet graph if any only if they differ by a constant. Let us also point out that the linear extension, to $\mathbb R^V$, of the layering $l$ serves as a conormal vector for the facet which corresponds to $l$.


    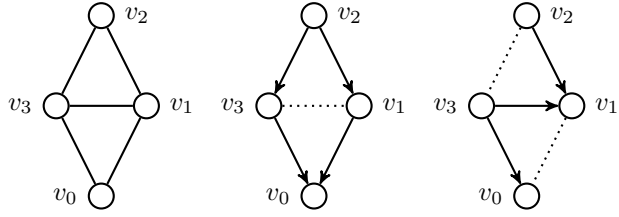
\begin{figure}
    	\begin{center}    
    		\begin{tikzpicture}[-,>=stealth',auto,scale=0.6,
    		thick]
    		\tikzstyle{o}=[circle,draw]
    		\node[o, label=left:$v_{0}$] (1) at (0, 0) {};
    		\node[o, label=left:$v_{3}$] (2) at (-1, 2) {};
    		\node[o, label=right:$v_{1}$] (3) at (1, 2) {};
    		\node[o, label=right:$v_{2}$] (4) at (0, 4) {};
    		\path[-,every node/.style={font=\sffamily\small}]
    		(2) edge node [above] {} (1);
    		\path[-,,every node/.style={font=\sffamily\small}]
    		(3) edge node [below] {} (1)
    		(4) edge node [above] {} (2)
    		(4) edge node [below] {} (3);
    		\path[-,every node/.style={font=\sffamily\small}]
    		(2) edge node [below] {} (3);
    		\end{tikzpicture}
    		\begin{tikzpicture}[-,>=stealth',auto,scale=0.6,
    		thick]
    		\tikzstyle{o}=[circle,draw]
    		\node[o, label=left:$v_{0}$] (1) at (0, 0) {};
    		\node[o, label=left:$v_{3}$] (2) at (-1, 2) {};
    		\node[o, label=right:$v_{1}$] (3) at (1, 2) {};
    		\node[o, label=right:$v_{2}$] (4) at (0, 4) {};
    		\path[->,every node/.style={font=\sffamily\small}]
    		(2) edge node [above] {} (1);
    		\path[->,,every node/.style={font=\sffamily\small}]
    		(3) edge node [below] {} (1)
    		(4) edge node [above] {} (2)
    		(4) edge node [below] {} (3);
    		\path[dotted,every node/.style={font=\sffamily\small}]
    		(2) edge node [below] {} (3);
    		\end{tikzpicture}
    		\begin{tikzpicture}[-,>=stealth',auto,scale=0.6,
    		thick]
    		\tikzstyle{o}=[circle,draw]
    		\node[o, label=left:$v_{0}$] (1) at (0, 0) {};
    		\node[o, label=left:$v_{3}$] (2) at (-1, 2) {};
    		\node[o, label=right:$v_{1}$] (3) at (1, 2) {};
    		\node[o, label=right:$v_{2}$] (4) at (0, 4) {};
    		\path[->,every node/.style={font=\sffamily\small}]
    		(2) edge node [above] {} (1)
    		(2) edge node [below] {} (3)
    		(4) edge node [below] {} (3);
    		\path[dotted,every node/.style={font=\sffamily\small}]
    		(4) edge node [above] {} (2)
    		(3) edge node [below] {} (1);
    		\end{tikzpicture}
    	\end{center}
    	\caption{An undirected graph and two of its subgraphs, with orientations. (They are in fact facet graphs.) Hidden edges are drawn by dotted lines.}
    	\label{fig:facet_graphs}
    \end{figure} 

In our paper \cite{semibalanced} we investigated exactly the polytopes seen in Theorem \ref{thm:facets_of_P_G}. Let us recall some definitions and results. 

\begin{defn}[Root polytope of a digraph]
	For a directed graph $G$, its \emph{root polytope} is defined as $\mathcal{Q}_G = \conv\{\mathbf v - \mathbf{u}\mid \overrightarrow{uv}\in E(G) \}$.
\end{defn}

With this definition, the facets of $P_\cG$ are exactly the root polytopes of the facet graphs of $\cG$. Let us say a bit more about these digraphs.

\begin{defn}[Semi-balanced digraph]
	A digraph $G$ is called \emph{semi-balanced} if for each cycle $C$, the numbers of the edges of $C$ pointing in the two cyclic directions are the same.
\end{defn}

We recall a characterization of semi-balanced digraphs from \cite{semibalanced}.

\begin{thm}\cite{semibalanced}\label{thm:semibalanced_graphs_char}
 $G$ is semi-balanced if and only if there is a layering $l\colon V \to \mathbb{Z}$ so that regarding the orientation of $G$, we have $l(h)-l(t)=1$ for each edge $\overrightarrow{th}$.
\end{thm}

\begin{cor}
	Any facet of $P_\cG$ 
	is the root polytope of a spanning subgraph of $\cG$ that is oriented in a semi-balanced way. 
	If $\cG$ is a bipartite graph, then the facets of $P_\cG$ are in bijection with the semi-balanced orientations of $\cG$.
\end{cor}

\begin{proof}
	The first statement follows directly from Theorems \ref{thm:facets_of_P_G} and \ref{thm:semibalanced_graphs_char}.
	
	Now suppose that $\cG$ is bipartite with vertex classes $U$ and $W$. Each facet has the form $\mathcal{Q}_{G_l}$ where we have $|l(u)-l(v)|=1$ for each $uv\in G_l$ and $G_l$ is a spanning subgraph of $\cG$. This ensures that either $l(u)$ is odd for each $u\in U$ and $l(w)$ is even for each $w\in W$, or vice versa. But this also implies that we cannot have $l(u)=l(v)$ for any $uv\in \cG$, whence $|l(u)-l(v)|=1$ for each edge $uv\in \cG$. That is, each edge of $\cG$ is present and receives an orientation in $G_l$.
	
	Conversely, it is clear from Theorem \ref{thm:facets_of_P_G} that a layering that corresponds to a semi-balanced orientation defines a facet.
\end{proof}

We will also need the following basic observations about the simplices within facets.

\begin{prop}\cite[Proposition 3.1]{semibalanced}
For a semibalanced digraph $G$, within $\mathcal{Q}_G$, some vertices are affinely independent if and only if they correspond to the edge set of a forest. In particular, spanning trees of $G$ give rise to maximal simplices in $\mathcal Q_G$.
\end{prop}

\begin{prop} \cite[Lemma 3.5]{semibalanced} \label{prop:unimodular_simplices}
For any facet graph $G$ and spanning tree $T\subset G$, the simplex $\tilde{\mathcal{Q}}_T=\conv\{0, \mathbf{v}-\mathbf{u} \mid \overrightarrow{uv}\in E(G)\}$ is unimodular.
\end{prop}

\subsection{Dissecting a facet of the symmetric edge polytope: Jaeger trees}

In \cite{semibalanced} we gave a method for dissecting the root polytope of a semi-balanced digraph into simplices, in a shellable way. This was done using the so-called Jaeger trees of the digraph.
Let us recall some details.

Let $G$ be a semi-balanced digraph.
To define Jaeger trees of $G$, we need to fix a ribbon structure for $G$. 
Here the notion of a \emph{ribbon structure} is independent of the graph's orientation: it is a family of cyclic permutations, namely for each vertex $x$ of $G$, a cyclic permutation of the edges incident to $x$ (the collection of in- and out-edges) is given.
For an edge $xy$ of $G$, we use the following notations: 
\begin{itemize}
	\item $yx^+_G$: the edge following $yx$ at $y$
	\item $xy^+_G$: the edge following $xy$ at $x$.
\end{itemize}
If $G$ is clear from the context, we omit the subscript.

In our applications, $G$ will often be a facet graph of an undirected graph $\cG$. In this case, by fixing a ribbon structure for $\cG$, all facet graphs $G$ inherit a ribbon structure from $\cG$ in the natural way.

In addition to the ribbon structure, we also need to fix a \emph{basis} $(b_0,b_0b_1)$, where the \emph{base node} $b_0$ is an arbitrary vertex of $G$ and the \emph{base edge} $b_0b_1$ is an arbitrary edge incident to $b_0$ in $G$. 
Note that no assumption is made about the orientation of $b_0b_1$.

Suppose that a ribbon structure and a basis are fixed. Then any spanning tree $T$ of $G$ gives us a natural ``walk'' in the graph $G$. This was defined by Bernardi \cite{Bernardi_first}, and following him we call it the \emph{tour of $T$}. The following definition is valid for both directed and undirected graphs, as orientations play no role in it.

\begin{defn}[Tour of a tree] \label{def:tour_of_a_tree}
	Let $\cG$ be a ribbon graph with a basis $(b_0,b_0b_1)$, and let $T$ be a spanning tree of $\cG$.
	The tour of $T$ is a sequence of node-edge pairs, starting with $(b_0, b_0b_1)$. If the current node-edge pair is $(x,xy)$ and $xy\notin T$, then the current node-edge pair of the next step is $(x,xy^+_{\cG})$. If the current node-edge pair is $(x,xy)$ and $xy\in T$, then the current node-edge pair of the next step is $(y,yx^+_\cG)$. In the first case we say that the tour \emph{skips} $xy$ and in the second case we say that the tour \emph{traverses} $xy$. The tour stops right before when $(b_0,b_0b_1)$ would once again become the current node-edge pair. 
\end{defn}

Bernardi proved {\cite[Lemma 5]{Bernardi_first}} that in the tour of a spanning tree $T$, each edge $xy$ of $G$ becomes current edge twice, in one case with $x$ as current node, and in the other case with $y$ as current node.

\begin{ex}\label{ex:tour}
	Figure \ref{fig:tour_and_Jaeger_ex} shows two spanning trees in a semi-balanced digraph. Let the ribbon structure be induced by the positive orientation of the plane and let the basis be $(v_0,v_0v_3)$. 
	
	The tour of the tree in the left panel is $(v_0, v_0v_3)$, $(v_3,v_3v_1)$, $(v_1,v_1v_4)$, $(v_1, v_1v_6)$, $(v_6,v_6v_2), (v_6,v_6v_0), (v_6,v_6v_1), (v_1, v_1v_3), (v_3, v_3v_0)$, $(v_0, v_0v_6)$, $(v_0, v_0v_5)$, $(v_5, v_5v_2)$, $(v_2, v_2v_6)$, $(v_2, v_2v_4)$, $(v_4, v_4v_1)$, $(v_4, v_4v_2)$, $(v_2, v_2v_5)$, $(v_5, v_5v_0)$.
	
	The tour of the tree on the right is $(v_0, v_0v_3)$, $(v_0,v_0v_6)$, $(v_0,v_0v_5)$, $(v_5, v_5v_2)$, $(v_2,v_2v_6), (v_6,v_6v_0), (v_6,v_6v_1), (v_1, v_1v_3), (v_3, v_3v_0)$, $(v_3, v_3v_1)$, $(v_1, v_1v_4)$, $(v_1, v_1v_6)$, $(v_6, v_6v_2)$, $(v_2, v_2v_4)$, $(v_4, v_4v_1)$, $(v_4, v_4v_2)$, $(v_2, v_2v_5)$, $(v_5, v_5v_0)$.
\end{ex}

The following is the key notion in this paper. Note that this is the point where orientations start to matter.

\begin{defn}[Jaeger tree]
\label{def:jaegertree}
    Let $G$ be a 
    digraph.
	Given a fixed ribbon structure and basis for $G$, we call a spanning tree $T$ of $G$ \emph{Jaeger tree} if for each edge $\overrightarrow{th}\in G-T$, in the tour of $T$, the pair $(t,th)$ becomes current node-edge pair before $(h,th)$. In other words, in the tour of $T$, each non-edge is first seen at its tail.
\end{defn}

\begin{figure}[t]
	\begin{center}
		\begin{tikzpicture}[-,>=stealth',auto,scale=0.4]
		\tikzstyle{e}=[{circle,draw,scale=0.6}]
		\tikzstyle{v}=[{circle,draw,scale=0.6}]
		\tikzstyle{u}=[{circle}]
		\node[e] (3) at (0, 0.9) {};
		\node[v] (1) at (3, 2.3) {};
		\node[v] (0) at (-3, 2.3) {};
		\node[e] (4) at (3, 4.7) {};
		\node[e] (5) at (-3, 4.7) {};
		\node[v] (2) at (0, 6.2) {};
		\node[e] (6) at (0, 3.6) {};
		
		\node[u] at (-0.9, 0.6) {$v_3$};
		\node[u] at (3.8, 2.3) {$v_1$};
		\node[u] at (-3.8, 2.3) {$v_0$};
		\node[u] at (3.8, 4.7) {$v_4$};
		\node[u] at (-3.8, 4.7) {$v_5$};
		\node[u] at (0.9, 6.4) {$v_2$};
		\node[u] at (-0.9, 4) {$v_6$};
		
		\node[{circle,fill,color=gray,scale=0.4}] at (-3.4, 1.8) {};
	    \draw [->,color=gray] (-3.4, 1.8) -- (-2.4, 1.35);
	    
		\path[->,line width=0.6mm]
		(0) edge node {} (3)
		(0) edge node {} (5)
		(1) edge node {} (3)
		(2) edge node {} (4)
		(1) edge node {} (6)
		(2) edge node {} (5);
		\path[->,dashed,line width=0.2mm]
		(0) edge node {} (6)
		(1) edge node {} (4)
		(2) edge node {} (6);
		\end{tikzpicture}
		\hspace{0.6cm}
		\begin{tikzpicture}[-,>=stealth',auto,scale=0.4]
		\tikzstyle{e}=[{circle,draw,scale=0.6}]
		\tikzstyle{v}=[{circle,draw,scale=0.6}]
		\tikzstyle{u}=[{circle}]
		\node[e] (3) at (0, 0.9) {};
		\node[v] (1) at (3, 2.3) {};
		\node[v] (0) at (-3, 2.3) {};
		\node[e] (4) at (3, 4.7) {};
		\node[e] (5) at (-3, 4.7) {};
		\node[v] (2) at (0, 6.2) {};
		\node[e] (6) at (0, 3.6) {};
		
		\node[u] at (-0.9, 0.6) {$v_3$};
		\node[u] at (3.8, 2.3) {$v_1$};
		\node[u] at (-3.8, 2.3) {$v_0$};
		\node[u] at (3.8, 4.7) {$v_4$};
		\node[u] at (-3.8, 4.7) {$v_5$};
		\node[u] at (0.9, 6.4) {$v_2$};
		\node[u] at (-0.9, 4) {$v_6$};
	
		\node[{circle,fill,color=gray,scale=0.4}] at (-3.4, 1.8) {};
	    \draw [->,color=gray] (-3.4, 1.8) -- (-2.4, 1.35);
		\path[->,line width=0.6mm]
		(0) edge node {} (5)
		(1) edge node {} (3)
		(1) edge node {} (6)
		(2) edge node {} (4)
		(2) edge node {} (5)
		(2) edge node {} (6);
		\path[->,dashed,line width=0.2mm]
		(0) edge node {} (6)
		(0) edge node {} (3)
		(1) edge node {} (4);
		\end{tikzpicture}
	\end{center}
	\caption{Two spanning trees in a semi-balanced digraph.
	The ribbon structure is induced by the positive orientation of the plane and the basis is indicated by the small gray arrow. The tree on the left is not a Jaeger tree, while the one on the right is.
	See Examples \ref{ex:tour} and \ref{ex:Jaeger} for more details.
	}
	\label{fig:tour_and_Jaeger_ex}
\end{figure}
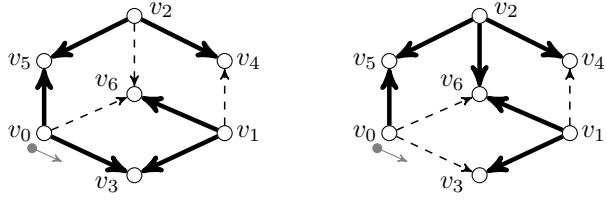

\begin{ex}\label{ex:Jaeger}
    The tree of the left panel of Figure \ref{fig:tour_and_Jaeger_ex} is not a Jaeger tree, because in its tour $(v_6,v_6v_2)$ precedes $(v_2,v_2v_6)$ whereas the edge $v_2v_6$ is oriented from $v_2$ to $v_6$ (and similarly for $v_0v_6$). On the other hand, the tree in the right panel is a Jaeger tree.
\end{ex}


Note that for a Jaeger tree, we require that the tail of a non-tree edge be seen before its head in the tour of $T$, but we do not care about the orientations of tree edges. However, when analyzing relationships of Jaeger trees, it will turn out that the latter can also be important. Hence we introduce the following terminology:

\begin{defn}
    Let $T$ be a spanning tree of a ribbon digraph $G$, and let $b_0\in V(G)$. We say that $\overrightarrow{th}\in T$ is a \emph{tail-edge} of $T$ if $b_0$ and the tail $t$ of $\overrightarrow{th}$ are in the same component of $T-\overrightarrow{th}$. We call $\overrightarrow{th}$ a \emph{head-edge} otherwise.
\end{defn}

Notice that $\overrightarrow{th}$ being a tail-edge is equivalent to the fact that (for an arbitrary ribbon structure and basis where $b_0$ is the base node) $(t,th)$ comes before $(h,th)$ in the tour of $T$. However, we can define tail-edges without refering to a ribbon structure, and fixing only the base node.

\begin{ex}
	The right panel of Figure \ref{fig:tour_and_Jaeger_ex} shows a Jaeger tree (for the given ribbon structure and basis).
	The tail edges of the tree are $v_0v_5$, $v_2v_4$, $v_2v_6$ and $v_1v_3$ while the head-edges are $v_5v_2$ and $v_6v_1$.
\end{ex}

The following property makes Jaeger trees very useful for us. 

\begin{thm}\label{thm:semibalanced_dissection}\cite{semibalanced}
	For a semi-balanced digraph $G$, the simplices corresponding to Jaeger trees dissect the root polytope $\mathcal{Q}_G$.
\end{thm}

Hence, if we take the Jaeger trees for each facet graph of $\cG$ (for some ribbon structure and basis, which might 
even vary from facet to
facet) and add the origin to each 
corresponding simplex as a new vertex, then we get a dissection of the symmetric edge polytope $P_\cG$. Since by Proposition \ref{prop:unimodular_simplices} these simplices are all unimodular, we immediately obtain that the normalized volume of $P_\cG$ is equal to the sum of the numbers of Jaeger trees over all facet graphs. (Note that, also by Thorem \ref{thm:semibalanced_dissection}, the number of Jaeger trees of a semibalanced digraph is independent of the chosen ribbon structure and basis.)

Jaeger trees of a semi-balanced digraph have a property that further enhances their usefulness: 
The dissection by Jaeger trees is shellable, with the following shelling order.  



\begin{defn}\label{def:ordering_of_Jaeger_trees_in_a_face} 
Let $T$ and $T'$ be Jaeger trees of a semi-balanced digraph $G$ (for some fixed ribbon structure and basis). We say that $T' \prec T$ if their respective tours first differ so that the current node-edge pair $(t,\overrightarrow{th})$ satisfies $\overrightarrow{th}\notin T'$ and $\overrightarrow{th} \in T$. (I.e., an edge seen at its tail is not included in $T'$ but is included in $T$.)
\end{defn}

Note that as we supposed that $T$ and $T'$ were both Jaeger trees, the first difference of their tours necessarily looks as above and hence this is a complete ordering. In fact, the order has a natural extension to all spanning  trees \cite{semibalanced}. 

\begin{thm}\cite{semibalanced} \label{thm:shelling_for_semibalanced}
	The order $\prec$ is a shelling order of the dissection given by Jaeger trees. For each Jaeger tree $T$, the number of facets of the corresponding simplex $Q_{T}$ that lie in the union of previous simplices, 
	equals the number of tail-edges $e\in T$ such that the fundamental cut $C^*(T,e)$ is not a directed cut.
\end{thm}

\begin{remark}
We remark that in \cite{semibalanced} those tail-edges of a Jaeger tree $T$ whose fundamental cuts are not directed are called the internally semipassive edges of $T$.
This relates to an activity-like notion defined there; moreover, 
\cite[Lemma 6.4]{semibalanced} provides another equivalent descriptions of these edges.
Namely, they are exactly the edges that arise as a ``first difference'' between the tours of $T$ and another Jaeger tree. 
\end{remark}


In the next section, we show that the shellings of the facets can be put together into a shelling of the whole boundary (or, by coning over the origin, to a shelling of the dissection of $P_\cG$). Moreover, we will be able to use this shelling to obtain a formula for the $h^*$-polynomial of $P_\cG$.

\section{Shellability and $h^*$-vector}\label{sec:shelling_orders}

In this section we give a shellable dissection into unimodular simplices for an arbitrary symmetric edge polytope, and determine the $h$-vector of the shelling (and hence, by Proposition \ref{prop:h-dissect}, also the $h^*$-vector of the polytope).

We have already pointed out that (for an arbitrary fixed ribbon structure) Jaeger trees yield dissections of the facets of the symmetric edge polytope, moreover, for each facet, this dissection is shellable. By putting these dissections together we get a dissection of the whole surface of $P_\cG$. 
In this section, we show that this dissection of the boundary of $P_\cG$ is also shellable.

Note that for any dissection of the boundary of $P_\cG$, we can add the origin to each simplex of the dissection, which results in a dissection for $P_\cG$. This dissection of $P_\cG$ is shellable if and only if the dissection of the boundary was shellable. Also, notice that the $h$-vectors coincide for the shelling of the boundary and for the shelling of $P_\cG$.
We will focus on the shellability of the dissection of the boundary.

In fact we will give two shelling orders. 
In the first one, our strategy will be to first choose an ordering of the facets, then build up the facets one-by-one in the chosen order. This will be a relatively flexible construction, where we do not need to use the same ribbon structure for different facet graphs. The only important point will be to use the same base point for each of them.

In the second shelling order, we use the same ribbon structure for all facet graphs, and give a recipe to directly compare Jaeger trees of different facets. This second type of shelling might start to build up a facet before some other one is finished.

\subsection{The face-by-face shelling}
Fix a vertex $b_0$ of $\cG$. For each facet graph $G$, take some ribbon structure and basis, such that the base point is $b_0$. (The ribbon structure and the base edge might be different for different facet graphs.)

Consider a real-valued weight function $f$ on the vertices of $\cG$ that associates $f(b_0)=1$ to the base vertex, while for $v\in V-b_0$, we let $-1\leq f(v)\leq 0$ in such a way that $\sum_{v\in V-b_0} f(v)=-1$ and the values $\{f(v)\mid v\in V-b_0\}$ are linearly independent over $\mathbb{Z}$.

Recall that for each facet graph $G$, there is a layering $l$ such that $l(v)-l(u)=1$ for every $\overrightarrow{uv}\in G$ and $|l(u)-l(v)|\leq 1$ for every $uv\in \cG$, moreover, two layerings $l$ and $l'$ define the same facet if and only if $l'(v)=l(v)+c$, with some constant $c$, for each vertex $v$. Let \[f(G)=\sum_{v\in V}l(v)\cdot f(v).\] 
Notice that as $\sum_{v\in V}f(v)=0$, this value is the same for any layering defining a given facet. 

We arrange the facet graphs in decreasing order according to this valuation. 
That is, let us label the facet graphs $G_1, G_2, \dots, G_M$ in such a way that $f(G_1) > f(G_2) > \dots> f(G_M)$.

\begin{remark}
Geometrically, $f$ can be viewed as a covector that is a generic perturbation of the dual of the vector $\mathbf b'_0$, which in turn is the projection of the standard basis vector $\mathbf b_0$ on the hyperplane of $P_\cG$. Each valuation $f(G)$ is the inner product of $f$ with the conormal of the facet which belongs to $G$. That is, facets are ordered in the following manner: We start from the origin and travel along a generic straight line specified by $f$. We list the facets as our trajectory crosses their planes; e.g., $G_1$ is the graph of the facet through which we first cross the boundary of $P_\cG$. By the time we `reach infinity,' half of the facets are recorded. Then we travel along the other half-line of our straight line in the same direction, that is, this time toward the origin, and continue the process. As a result, the symmetric pairs of the already-recorded facets come up in the opposite order. This is well known to be a shelling order of the facets. Our challenge is to combine it with the dissections of the facets and to keep track of the $h$-vector.
\end{remark}

For each facet graph $G_i$ we denote the set of Jaeger trees of $G_i$, for the chosen ribbon structure and basis, by $\jaeg(G_i)$. We also put $\jaeg(\cG)=\bigcup_{G \text{ is a facet graph of $\cG$}}\jaeg(G)$.
Here it is somewhat important to take Jaeger trees as sets of directed edges; without that extra information, the same tree can have  the Jaeger property with respect to multiple facet graphs. 
In this subsection we order the trees in $\jaeg(\cG)=\jaeg(G_1)\cup \jaeg(G_2)\cup \dots\cup \jaeg(G_M)$ in the following way. 

\begin{defn}[face-by-face ordering of Jaeger trees, $<_f$]
For $T\in \jaeg(G_i)$ and $T'\in \jaeg(G_j)$ we have $T<_f T'$ if and only if either $i<j$ or $i=j$ and $T\prec T'$ in the ordering of Jaeger trees of $G_i$ given in Definition \ref{def:ordering_of_Jaeger_trees_in_a_face}. 
\end{defn}

\begin{thm}\label{thm:shelling_of_symm_edge_poly}
	The simplices corresponding to $\mathcal{J}(\cG)$ form a shellable dissection of the boundary of $P_\cG$, with shelling order $<_f$.
	The terms of the resulting $h$-vector are $$h_i=|\{T\in\jaeg(\cG)\mid T  \text{ has exactly $i$ tail-edges}\}|.$$
\end{thm}

Before giving a proof, let us point out a corollary for the $h^*$-vector of $P_\cG$.

\begin{thm}\label{cor:h^*_of_sym_edge_poly}
$$(h^*_{P_\cG})_i = |\{T\in\jaeg(\cG)\mid T  \text{ has exactly $i$ tail-edges}\}|.$$
\end{thm}

\begin{proof}
Consider the collection $\{\tilde{\mathcal{Q}}_T\mid T\in \jaeg(\cG)\}$. It results from coning over a dissection of the boundary. Hence by Theorem \ref{thm:shelling_of_symm_edge_poly} this is a shellable dissection whose $h$-vector is as specified in Theorem \ref{thm:shelling_of_symm_edge_poly}. By Proposition \ref{prop:unimodular_simplices} the simplices are unimodular. Hence by Proposition \ref{prop:h-dissect}, the $h^*$-vector agrees with the $h$-vector.
\end{proof}

To prove Theorem \ref{thm:shelling_of_symm_edge_poly}, first we need two lemmas on the relationship of facets.

\begin{lemma}\label{l:reversing_cut_pointing_away}
    If $(V_0,V_1)$ is a directed cut in the facet graph $G$ such that $b_0\in V_0$ and all edges point from $V_0$ to $V_1$, then there exists another facet graph $G'$ such that $f(G')>f(G)$, moreover, edges of $G$ outside the cut $(V_0,V_1)$ are present in $G'$ with the same orientation as in $G$ (that is, each edge $uv\in\cG$ with either $u,v\in V_0$ or $u,v\in V_1$ is either not present in $G$, or it is present in $G'$ with the same orientation as in $G$), and
    $(V_0,V_1)$ is also a directed cut in $G'$, but in $G'$ each edge points from $V_1$ to $V_0$.
\end{lemma}

We note that $G$ and $G'$ above may not share the same edge set, not even when orientations are ignored. Thus as sets of edges, the cuts in $G$ and in $G'$ that are induced by the same partition $(V_0,V_1)$ of $V$, may be different.

\begin{proof}
	Let $l$ be a layering for $G$.
	Suppose that $G[V_1]$ has $k$ connected components ($k$ might be 1), and let the vertex sets of these components be $V_{1,1}, \dots, V_{1,k}$ (with $V_1=V_{1,1}\sqcup \dots \sqcup V_{1,k}$).
	Without loss of generality, we can suppose that there is some $r$ with $0\leq r \leq k$ such that for $i=1, \dots, r$ there is an edge $x_iy_i\in \cG$ with $x_i\in V_0$, $y_i\in V_{1,i}$, and $l(x_i)=l(y_i)$ (that is, an edge of $\cG$ that is not present in $G$), and for $i>r$, there is no such edge.
	
	Take the layering $l'$ defined by 
	$$l'(v) = \left\{\begin{array}{cl} 
	l(v)+1 & \text{if $v\in V_0$},  \\
	l(v) & \text{if $v\in V_{1,i}$ for $i\leq r$}, \\
	l(v)-1 & \text{if $v\in V_{1,i}$ for $i > r$}.
	\end{array} \right.
	$$
	We claim that $l'$ gives a facet graph $G'$ as claimed. We simultaneously show that $l'$ 
	satisfies the conditions of Theorem \ref{thm:facets_of_P_G}
	and that $G'$ satisfies the conditions of the lemma. If $uv\in \cG$ is an edge with either $u,v\in V_0$, $u,v\in V_{1,1} \cup \dots \cup V_{1,r}$, or $u,v\in V_{1,r+1} \cup \dots \cup V_{1,k}$, then $l'(v)-l'(u)=l(v)-l(u)$. 
	If $v\in V_{1,1} \cup \dots \cup V_{1,r}$ and $u\in V_{1,r+1}\cup \dots \cup V_{1,k}$ then $l'(v)-l'(u)=l(v)-l(u)+1$, but (since the $G[V_{1,i}]$ are connected components) we had $l(v)-l(u)=0$. This implies that the $l'$-difference for edges outside the cut $(V_0,V_1)$ is at most 1, moreover, while we may gain new edges in $G'$, the orientation of each edge of $G$ outside the cut remains the same.
	
	If $u\in V_0$ and $v\in V_{1,i}$ with $i\leq r$, then we had either $l(v)-l(u)=0$ or $l(v)-l(u)=1$ as each edge of the cut $(V_0,V_1)$ pointed toward $V_1$ (if it was present in $G$). Hence for these edges, we have $l'(v)-l'(u)=-1$ if $uv$ was not present in $G$ (and we supposed that there is at least one such edge) and $l'(v)-l'(u)=0$ if $uv$ was present in $G$. Thus, each such edge has layer-difference at most $1$ with respect to $l'$, and each edge of $G'$ between $V_0$ and $V_{1,i}$ points toward $V_0$.
	
	If $uv$ is an edge of $\cG$ so that $u\in V_0$ and $v\in V_{1,i}$ with $i>r$, then we had $l(v)-l(u)=1$ by the definition of $r$. Hence for these edges, we have $l'(v)-l'(u)=-1$. Thus, each such edge has $l'$-difference at most 1, and each edge of $G'$ between $V_0$ and $V_{1,i}$ points toward $V_0$.
	
	To show that $l'$ 
	gives a facet graph $G'$,
	it remains to show that $G'$ is connected. The induced subgraphs $G'[V_0], G'[V_{1,1}], \dots, G'[V_{1,k}]$ remain connected, as the edges within them did not change. For $i\leq r$, the graph $G'[V_{1,i}]$ stays connected to $G'[V_0]$ through the edge $x_iy_i$. For $i>r$, all edges between $V_0$ and $V_{1,i}$ remain in $G'$, hence $G'[V_{1,i}]$ is also connected to $G'[V_0]$ for $i>r$.

    Finally we need to ascertain that $f(G')>f(G)$, but this follows from the definition of $f$ and the obvious $f(G')=f(G)+\sum_{v\in V_0}f(v)-\sum_{v\in V_{1,r+1}\cup \dots \cup V_{1,k}}f(v)$: as $b_0\in V_0$, we have $\sum_{v\in V_0}f(v)>0$ and $\sum_{v\in V_{1,r+1}\cup \dots \cup V_{1,k}}f(v)<0$.
\end{proof}

\begin{lemma}\label{l:nonfirst_orientation_exits_dircut}
	If $G_i$ and $G_j$ are facet graphs with $i>j$ (that is, $f(G_i)<f(G_j)$), then there exists a cut $(V_0,V_1)$ such that $b_0\in V_0$ and each edge points from $V_1$ to $V_0$ in $G_j$ while each edge points from $V_0$ to $V_1$ in $G_i$. 
\end{lemma}

\begin{proof}
	Let $l_j$ be the layering of $G_j$ and $l_i$ be the layering of $G_i$. Shift the layerings so that $l_j(v)\geq l_i(v)$ for each $v\in V$, but there exists some $v\in V$ with $l_i(v)=l_j(v)$ (this is achievable for any two layerings). Then the difference $l_j(v)-l_i(v)$ is a nonnegative integer for each $v\in V$. Let $U(k)=\{v\in V\mid l_j(v)-l_i(v)=k\}$.
	
	For each edge $uv$ of $\cG$, we have $|l_i(u)-l_i(v)|\leq 1$ and similarly for $l_j$, hence $|(l_j(v)-l_i(v))-(l_j(u)-l_i(u))|=|(l_j(v)-l_j(u))+ (l_i(u)-l_i(v))|\leq 2$. Thus, any edge of $\cG$ connects vertices within some $U(k)$, or between $U(k)$ and $U(k+1)$, or between $U(k)$ and $U(k+2)$ for some $k$. Also, any edge of $\cG$ between $U(k)$ and $U(k+2)$ points toward $U(k+2)$ in $G_j$ and toward $U(k)$ in $G_i$. Similarly, any edge of $\cG$ between $U(k)$ and $U(k+1)$ is either not in $G_j$ and points toward $U(k)$ in $G_i$, or it points toward $U(k+1)$ in $G_j$ and it is not present in $G_i$. Hence for any $k\geq 0$, the partition $(U(0) \sqcup \dots \sqcup U(k), U(k+1)\sqcup U(k+2)\sqcup\dots )$ induces an oriented cut in both $G_i$ and $G_j$, with each edge oriented toward $U(0)\sqcup \dots \sqcup U(k)$ in $G_i$ and each edge oriented toward $U(k+1)\sqcup U(k+2)\sqcup \dots$ in $G_j$. 
	
	We claim that at least one of these cuts satisfies the condition of the Lemma. 
	For this, it is enough to show that $b_0\not\in U(0)$. Indeed, if $b_0\in U(k)$ for $k\neq 0$, then the sets $V_0=U(k)\sqcup U(k+1)\sqcup \dots$ and $V_1=U(0)\sqcup \dots \sqcup U(k-1)$ provide the desired cut.
	
	Now notice that if we had $b_0\in U(0)$, then the only vertex with $f>0$ would have the same coefficient in $f(G_i)$ and $f(G_j)$, while some vertices with negative $f$-value would have a larger coefficient in $f(G_j)$ than in $f(G_i)$. This would imply $f(G_j)<f(G_i)$ and thus contradict our assumption.
\end{proof}

\begin{remark}
Although we will not require it later, let us point out that Lemma \ref{l:nonfirst_orientation_exits_dircut} provides an explicit description of the first facet graph $G_1$: it is defined by the layering $l(v)=-\dist(b_0,v)$, where $\dist(b_0,v)$ means the minimal number of edges in a path between $b_0$ and $v$ in the (undirected) graph $\cG$.
    
    The function $l$ clearly gives us a layering, that is, the difference between the endpoints of any edge is at most $1$. Moreover, we cannot have a directed cut $(V_0,V_1)$ with $b_0\in V_0$ and each edge oriented toward $V_1$, since for any vertex $v\in V_1$, the edges of a shortest path to $b_0$ all point toward $b_0$. However, by Lemma \ref{l:nonfirst_orientation_exits_dircut}, if this were not the first facet graph then we would need to have such a cut.
\end{remark}

\begin{proof}[Proof of Theorem \ref{thm:shelling_of_symm_edge_poly}] 
Take an arbitrary Jaeger tree $T$ that is not the first Jaeger tree in $\jaeg(\cG)$ according to $<_f$.
We need to prove that the simplex $\mathcal{Q}_T$ meets $\bigcup_{T'<_f T}\mathcal{Q}_{T'}$ in the union of the facets $\mathcal{Q}_{T-e}$, where $e$ ranges over the tail-edges of $T$. First, we will show that the facet $\mathcal{Q}_{T-e}$ is a subset of $\bigcup_{T'<T}\mathcal{Q}_{T'}$ when $e$ is a tail-edge of $T$. 
Then we will show that if $T$ is not the first Jaeger tree, then the simplex $\mathcal{Q}_T$ is not disjoint from $\bigcup_{T'<_f T}\mathcal{Q}_{T'}$, and any point $\mathbf p\in \mathcal{Q}_T \cap (\bigcup_{T'<_f T}\mathcal{Q}_{T'})$ is in $\mathcal{Q}_{T-e}$ for some tail-edge $e$ of $T$. 

Let $G_i$ be the facet graph that contains $T$ as a Jaeger tree.
If some $e\in T$ is a tail-edge of $T$ and $C^*(T,e)$ is not an oriented cut, then by Theorem \ref{thm:shelling_for_semibalanced}, 
$$
\mathcal{Q}_{T-e}\subseteq \bigcup_{T'\in\jaeg(G_i):\, T' \prec T} \mathcal{Q}_{T'}\subseteq \bigcup_{T'<_f T}\mathcal{Q}_{T'}.
$$

If $e\in T$ is a tail-edge and $C^*(T,e)$ is an oriented cut, then $C^*(T,e)$ points away from its side containing $b_0$. 
Now by Lemma \ref{l:reversing_cut_pointing_away}, there exists another facet graph $G_j$ with $f(G_j)>f(G_i)$ (and hence $j<i$) such that all edges of $G_i$ outside $C^*(T,e)$ are in $G_j$ with the same orientation. 
Hence $\mathcal{Q}_{T-e}\subseteq Q_{G_i}\cap Q_{G_j}$ and in conclusion, $\mathcal{Q}_{T-e}\subseteq Q_{G_j}\subseteq \bigcup_{T'<_f T}\mathcal{Q}_{T'}$.

Next we claim that if $T$ is not the first Jaeger tree, then $\mathcal{Q}_T$ is not disjoint from $\bigcup_{T'<T}\mathcal{Q}_{T'}$. 
If $T$ is not the first Jaeger tree from its facet graph $G_i$, then $\mathcal{Q}_T$ intersects the union of the simplices of previous Jaeger trees of $G_i$ by Theorem \ref{thm:shelling_for_semibalanced} applied to $G_i$. If $T$ is the first Jaeger tree of its facet, then it is enough to show that there is a tail-edge $e\in T$ as we already proved that in this case $\mathcal{Q}_{T-e}\subseteq \bigcup_{T'<_f T}\mathcal{Q}_{T'}$.
If $T$ is the first Jaeger tree in its facet, but not the first tree altogether, then its facet is not the first facet. In this case, by Lemma \ref{l:nonfirst_orientation_exits_dircut}, there exists a directed cut $C^*$ in $G_i$ that points away from its side containing $b_0$. 
Now when we first take an edge from $C^*$ in the tour of $T$, it will be reached at its tail, as claimed.

Finally we show that any point $\mathbf p\in \mathcal{Q}_T \cap (\bigcup_{T'<_f T}\mathcal{Q}_{T'})$ is in $\mathcal{Q}_{T-e}$ for some tail-edge $e$ of $T$. 
Suppose that $\mathbf p\in \mathcal{Q}_T\cap \mathcal{Q}_{T'}$ where $T'<_f T$. If $T'\in\jaeg(G_i)$, then by Theorem \ref{thm:shelling_for_semibalanced}, there is a tail-edge $e\in T$ (moreover, $C(T,e)$ is not directed) such that $\mathbf{p}\in \mathcal{Q}_{T-e}$. 
If $T'\in \jaeg(G_j)$ for $j \neq i$, then necessarily $j<i$. By Lemma \ref{l:nonfirst_orientation_exits_dircut} there is a directed cut $(V_0,V_1)$ 
with $b_0\in V_0$ so that each edge points from $V_0$ to $V_1$ in $G_i$. Moreover, in $G_j$, each edge between $V_0$ and $V_1$ points toward $V_0$. Hence if $\mathbf{p}\in \mathcal{Q}_T \cap \mathcal{Q}_{T'}$ (that is, $\mathbf p$ is a convex linear combination of vectors representing edges in $T$, as well as those in $T'$), then the coefficient of each edge of the cut $(V_0,V_1)$ needs to be zero in $\mathbf{p}$. 
Let $e$ be the first edge of $T$ that we reach from the cut $(V_0,V_1)$ in the tour of $T$. As in $G_i$ each edge points away from $V_0$, we reach $e$ at its tail. 
As $e$ is in the cut, 
our earlier observation provides that
$\mathbf{p}\in \mathcal{Q}_{T-e}$.
\end{proof}

\subsection{The quadratic shelling}
We give a second type of shelling order for $P_\cG$, that we call the quadratic shelling order. In this shelling, we will directly compare Jaeger trees of different facet graphs based on their tours.

Fix a ribbon structure and a basis $(b_0,b_0b_1)$. This time we will use this data for each facet graph. In fact, assuming that we did the same for the face-by-face shelling order, then the two orders would not be dramatically different. As we are about to see, each simplex of the dissection attaches to the previous ones along the same set of facets; in particular, the contribution to  the $h$-vector of each simplex is the same with respect to either order. 

To facilitate the comparison of Jaeger trees of different facet graphs, in this section, when considering the tour of a tree $T$ of a facet graph $G$, we also keep track of the hidden edges. (See Example \ref{ex:<_4}.) In other words, 
we consider the tour of $T$ in $\cG$.
When defining Jaeger trees, we will simply disregard the hidden edges (or in other words, we do not care about which endpoint of a hidden edge we reached first). This way, the definition of a Jaeger tree does not change.


    \begin{figure}
    	\begin{center}    
    		\begin{tikzpicture}[-,>=stealth',auto,scale=0.6, thick]
    		\tikzstyle{o}=[circle,draw]
    		\node[o, label=left:$v_{0}$] (1) at (0, 0) {};
    		\node[o, label=left:$v_{3}$] (2) at (-1, 2) {};
    		\node[o, label=right:$v_{1}$] (3) at (1, 2) {};
    		\node[o, label=right:$v_{2}$] (4) at (0, 4) {};
    		\path[->, line width=0.2mm]
    		(2) edge node [above] {} (1);
    		\path[->, line width=0.65mm,every node/.style={font=\sffamily\small}]
    		(3) edge node [below] {} (1)
    		(4) edge node [above] {} (2)
    		(4) edge node [below] {} (3);
    		\path[dotted,every node/.style={font=\sffamily\small}]
    		(2) edge node [below] {} (3);
    		
	        \draw [fill=gray,color=gray] (0.6,-0.3) circle [radius=.05];
	        \draw [->,color=gray] (0.6,-0.3) -- (0.9, 0.4);
    		\end{tikzpicture}
    		\begin{tikzpicture}[-,>=stealth',auto,scale=0.6, thick]
    		\tikzstyle{o}=[circle,draw]
    		\node[o, label=left:$v_{0}$] (1) at (0, 0) {};
    		\node[o, label=left:$v_{3}$] (2) at (-1, 2) {};
    		\node[o, label=right:$v_{1}$] (3) at (1, 2) {};
    		\node[o, label=right:$v_{2}$] (4) at (0, 4) {};
    		\path[->,line width=0.65mm,every node/.style={font=\sffamily\small}]
    		(2) edge node [above] {} (1)
    		(2) edge node [below] {} (3)
    		(4) edge node [below] {} (3);
    		\path[dotted,every node/.style={font=\sffamily\small}]
    		(4) edge node [above] {} (2)
    		(3) edge node [below] {} (1);
    		\draw [fill=gray,color=gray] (0.6,-0.3) circle [radius=.05];
	        \draw [->,color=gray] (0.6,-0.3) -- (0.9, 0.4);
    		\end{tikzpicture}
    		\begin{tikzpicture}[-,>=stealth',auto,scale=0.6, thick]
    		\tikzstyle{o}=[circle,draw]
    		\node[o, label=left:$v_{0}$] (1) at (0, 0) {};
    		\node[o, label=left:$v_{3}$] (2) at (-1, 2) {};
    		\node[o, label=right:$v_{1}$] (3) at (1, 2) {};
    		\node[o, label=right:$v_{2}$] (4) at (0, 4) {};
    		\path[->,line width=0.2mm]
    		(1) edge node [below] {} (3);
    		\path[->,line width=0.65mm,every node/.style={font=\sffamily\small}]
    		(1) edge node [above] {} (2)
    		(2) edge node [above] {} (4)
    		(3) edge node [below] {} (4);
    		\path[dotted]
    		(2) edge node [below] {} (3);
    		\draw [fill=gray,color=gray] (0.6,-0.3) circle [radius=.05];
	        \draw [->,color=gray] (0.6,-0.3) -- (0.9, 0.4);
    		\end{tikzpicture}
    		\begin{tikzpicture}[-,>=stealth',auto,scale=0.6, thick]
    		\tikzstyle{o}=[circle,draw]
    		\node[o, label=left:$v_{0}$] (1) at (0, 0) {};
    		\node[o, label=left:$v_{3}$] (2) at (-1, 2) {};
    		\node[o, label=right:$v_{1}$] (3) at (1, 2) {};
    		\node[o, label=right:$v_{2}$] (4) at (0, 4) {};
    		\path[->,line width=0.2mm]
    		(1) edge node [above] {} (2);
    		\path[->,line width=0.65mm]
    		(1) edge node [below] {} (3)
    		(2) edge node [above] {} (4)
    		(3) edge node [below] {} (4);
    		\path[dotted]
    		(2) edge node [below] {} (3);
    		\draw [fill=gray,color=gray] (0.6,-0.3) circle [radius=.05];
	        \draw [->,color=gray] (0.6,-0.3) -- (0.9, 0.4);
    		\end{tikzpicture}
    	\end{center}
    	\caption{Jaeger trees in various facet graphs of a graph. The tree edges are thick. Hidden edges are dotted.
    	The basis is $(v_0,v_0v_1)$, which is indicated by the little gray arrow.}
    	\label{fig:<_4_Jaeger_trees}
    \end{figure}
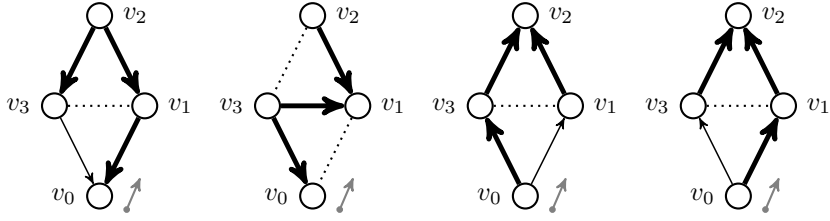 


For any two 
spanning trees $T$, $T'$
of a ribbon graph,
we say that the tours of $T$ and $T'$ agree until some point if up to that moment the current node-edge pairs of the two tours agree, moreover the status of each current edge has been the same up to this point, meaning that it was either hidden with respect to both trees, or was present in both and oriented in the same way. 

In this subsection we order the Jaeger trees in $\jaeg(\cG)$ the following way:

\begin{defn}[quadratic ordering of Jaeger trees, $<_4$]
Let $T_1,T_2,T_3,T_4\in\jaeg(\cG)$. If their tours agree until the node-edge pair $(u,uv)$ becomes current, but at that point $\overrightarrow{vu}\in T_1$, the edge $uv$ is not present in the facet graph of $T_2$, the oriented edge $\overrightarrow{uv}$ is in the facet graph of $T_3$ but $\overrightarrow{uv}\notin T_3$, finally $\overrightarrow{uv}\in T_4$, then we put $T_1 <_4 T_2 <_4 T_3 <_4 T_4$. 
\end{defn}

Note that as we look at first differences of Jaeger trees, it is not possible to have $\overrightarrow{vu}\notin T$ if $\overrightarrow{vu}$ is in the facet graph of $T$. Indeed, by the Jaeger property, in this case $(v,vu)$ would have to be current earlier than $(u,uv)$ but then that would have been an earlier difference.

\begin{ex}\label{ex:<_4}
Figure \ref{fig:<_4_Jaeger_trees} shows four Jaeger trees in various facet graphs, with ribbon structure induced by the positive orientation of the plane, and basis $(v_0,v_0v_1)$.
For the tree of the leftmost panel, its tour (considered in $\cG$) is $(v_0,v_0v_1)$, $(v_1,v_1v_2)$, $(v_2,v_2v_3)$, $(v_3,v_3v_0)$, $(v_3,v_3v_1)$, $(v_3,v_3v_2)$, $(v_2,v_2v_1)$, $(v_1,v_1v_3)$, $(v_1,v_1v_0)$, $(v_0,v_0v_3)$. 
The quadratic ordering $<_4$ arranges these Jaeger trees in increasing order from left to right. Indeed, the first difference between the tours of any two of the trees is at the node-edge pair $(v_0,v_0v_1)$, which is a head-edge for the leftmost tree, a hidden edge for the second one, a nonedge seen from its tail for the third one, and a tail edge for the rightmost one.
\end{ex}

In preparation to our main claim about the quadratic order, we need the following two statements.


\begin{lemma}\label{l:edge_in_or_out_in_two_facets}
	If there are two facet graphs $G$ and $G'$ such that an edge $xy$ is present in $G$ oriented as $\overrightarrow{xy}$ and it is not present in $G'$, then there is a cut $(V_0,V_1)$ such that $x\in V_0$, $y\in V_1$, moreover, in $G$ each edge between $V_0$ and $V_1$ is either not present or oriented from $V_0$ to $V_1$, and in $G'$ each edge between $V_0$ and $V_1$ is either not present or oriented from $V_1$ to $V_0$.
\end{lemma}

\begin{proof}
	Let $l$ be the layering of $G$ and $l'$ be the layering of $G'$, and shift them so that $l(y)=l'(y)$. Notice that in this case, $l'(x)-l(x)=1$. Just like in the proof of Lemma \ref{l:nonfirst_orientation_exits_dircut}, for $i\in\mathbb{Z}$ let $U(i)=\{v\in V\mid l'(v)-l(v)=i\}$. Take $V_0=\bigcup_{i\leq 0}U(i)$ and $V_1=\bigcup_{i\geq 1}U(i)$. As in the earlier proof, $(V_0,V_1)$ will be a cut, where in $G$ each edge between $V_0$ and $V_1$ is either not present or oriented from $V_0$ to $V_1$, and in $G'$ each edge between $V_0$ and $V_1$ is either not present or oriented from $V_1$ to $V_0$.
\end{proof}

\begin{lemma}\label{l:edge_ordered_differently_in_two_facets}
	If there are two facet graphs $G$ and $G'$ such that an edge $xy$ is oriented as $\overrightarrow{xy}$ in $G$ and it is oriented as $\overrightarrow{yx}$ in $G'$, then there is a cut $(V_0,V_1)$ such that $x\in V_0$, $y\in V_1$, moreover, in $G$ each edge between $V_0$ and $V_1$ is either not present or oriented from $V_0$ to $V_1$, and in $G'$ each edge between $V_0$ and $V_1$ is either not present or oriented from $V_1$ to $V_0$. 
\end{lemma}

\begin{proof}
	Let $l$ be the layering of $G$ and $l'$ be the layering of $G'$, and shift them so that $l(y)=l'(y)$. Notice that in this case, $l'(x)-l(x)=2$. Once again, for $i\in\mathbb{Z}$ let $U(i)=\{v\in V\mid l'(v)-l(v)=i\}$. Take $V_0=\bigcup_{i\leq 0}U(i)$ and $V_1=\bigcup_{i\geq 1}U(i)$. As in the previous proof, $(V_0,V_1)$ will be a cut, where in $G$ each edge between $V_0$ and $V_1$ is either not present or oriented from $V_0$ to $V_1$, and in $G'$ each edge between $V_0$ and $V_1$ is either not present or oriented from $V_1$ to $V_0$.
\end{proof}

\begin{thm}\label{t:quadratic_shelling}
For any connected graph $\cG$, the quadratic order $<_4$ is a shelling order on the set of simplices corresponding to the collection of Jaeger trees $\mathcal{J}(\cG)$. 
\end{thm}

As before, the terms of the resulting $h$-vector are $$h_i=|\{T\in\jaeg(\cG)\mid T  \text{ has exactly $i$ tail-edges}\}|.$$
The coincidence here with the formula of Theorem \ref{thm:shelling_of_symm_edge_poly} comes as no surprise in light of Proposition \ref{prop:h-dissect} and Corollary \ref{cor:h^*_of_sym_edge_poly}. Recall though that in this subsection $\mathcal J(\cG)$ is defined by using (restrictions of) the same ribbon structure, whereas in the case of Theorem \ref{thm:shelling_of_symm_edge_poly}, the definition was more flexible.

\begin{proof}
	
	The outline of the proof is the same as that of the proof of Theorem \ref{thm:shelling_of_symm_edge_poly}. First we show that if $e\in T$ is a tail-edge then $\mathcal{Q}_{T-e}\subseteq \bigcup_{T'<_4 T}\mathcal{Q}_{T'}$. 
	Then we show that if $T$ is not the first tree in $<_4$, then $\mathcal{Q}_T$ is not disjoint from $\bigcup_{T'<_4 T}\mathcal{Q}_{T'}$, moreover, if $\mathbf x\in \mathcal{Q}_T\cap (\bigcup_{T'<_4 T}\mathcal{Q}_{T'})$ then $\mathbf x\in \mathcal{Q}_{T-e}$ for some tail-edge $e\in T$.
	
	First, let us notice that if $T,T'\in \jaeg(G)$ for some facet graph $G$, then $T$ and $T'$ are ordered the same way in $<_4$ as in the ordering of Jaeger trees of $G$. Indeed, as they live in the same directed subgraph, the first difference between their tours must be that one of them contains an edge reached at its tail, while the other one does not. In this case, they are ordered the same way as for the ordering $\prec$ of $\jaeg(G)$.
	
	Let $T\in\jaeg(G)$.
	Suppose that $e\in T$ is a tail-edge and $C^*(T,e)$ is not directed. Then by Theorem \ref{thm:shelling_for_semibalanced}, $\mathcal{Q}_{T-e}\subseteq \bigcup_{ T'\in\jaeg(G), T' \prec T}\mathcal{Q}_{T'}\subseteq \bigcup_{T'<_4 T}\mathcal{Q}_{T'}$.
	
	Now suppose that $e\in T$ is a tail-edge but $C^*(T,e)$ is a directed cut. We show a set of Jaeger trees all preceding $T$ in $<_4$ such that the union of their simplices contains $\mathcal{Q}_{T-e}$.  
	Let the sides of the cut $C^*(T,e)$ be $V_0$ and $V_1$ with $b_0\in V_0$.
	
	Take the tour of $T$. As $T$ is a Jaeger tree, $e$ has its tail in $V_0$. Hence each edge of $C^*_G(T,e)$ has its tail in $V_0$. Thus, $e$ is the last edge of $C^*_G(T,e)$ to become current in the tour of $T$. 
	
	Take the facet graph $G'$ whose existence is proved in Lemma \ref{l:reversing_cut_pointing_away}. Notice that as in our case $C^*(T,e)$ is a fundamental cut, $G[V_1]$ is connected. Hence $G$ and $G'$ agree outside the cut $(V_0,V_1)$. If there is no edge in $\cG$ between $V_0$ and $V_1$ that is not present in $G$ then we obtained $G'$ from $G$ by reversing the cut $C^*_G(T,e)$. If there are edges of $\cG$ between $V_0$ and $V_1$ that are not present in $G$, then we get $G'$ from $G$ by removing the edges of $C^*_G(T,e)$ and adding the non-present edges of $\cG$ between $V_0$ and $V_1$ directed toward $V_0$.
	In the first case, let $uv$ be the edge of $C^*_G(T,e)$ that first becomes current in the tour of $T$. In the second case, let $uv$ be the first non-edge of $G$ between $V_0$ and $V_1$ that becomes current in the tour of $T$. In both cases, let $u\in V_0$ and $v\in V_1$. 
	In both cases, $uv$ is in $G'$ with orientation $\overrightarrow{vu}$. Let $vw$ be the first edge of $G'$ following $vu$ in the ribbon structure of $G'$ at $v$ that is not in $G'(V_0,V_1)$. 
	Let $T_0$ be the component of $T-e$ containing $b_0$, and let $T_1$ be the other component. 
	Let $\jaeg(G'[V_1])$ denote the set of Jaeger trees of $G'[V_1]$ with base $(v,vw)$, where the ribbon structure is inherited from $G'$. Notice the following:
	$$
	\mathcal{T}:=\{T_0\cup \overleftarrow{uv} \cup T'_1: T'_1\in\jaeg(G'[V_1])\}\subseteq \jaeg(G').
	$$
	Firstly, these are all trees within $G'$. Now we show that they are Jaeger.
	Indeed, until reaching $\overleftarrow{uv}$, our walk agrees with the tour of $T$, hence we do not cut any edge at its head. We need to traverse $\overrightarrow{vu}$ as we reach it at its head. After traversing $\overrightarrow{vu}$, we arrive at $G'[V_1]=G[V_1]$. When we encounter an edge of $G'(V_0,V_1)$, we can cut it, since now the tails are in the side of $G'[V_1]$. Otherwise we traverse a Jaeger tree, hence within $V_1$ we also cut the non-tree edges at their tail. Arriving back to $T_0$ we finish the traversal as in the tour of $T$ with the exception that the edges of $G'(V_0,V_1)$ are already cut from their other side.
	
	Notice that the above described trees all precede $T$ in $<_4$, as the first difference is one of the followings: 
	In case 1, if $uv\neq e$, we traverse an edge from the head direction in $T'$ and remove the edge (from the tail direction) in $T$. In case 1, if $uv=e$, we traverse an edge from the head direction in $T'$ and traverse the edge from the tail direction in $T$. 
	In case 2 we traverse an edge from the head direction in $T'$ and it is not present in the facet of $T$. 
	
	As $\bigcup_{T'\in\jaeg(G'[V_1])} \mathcal{Q}_{T'}=Q_{G'[V_1]}=Q_{G[V_1]}\supseteq \mathcal{Q}_{T_1}$, we have $\mathcal{Q}_{T-e}\subseteq \bigcup_{T'\in\mathcal{T}} \mathcal{Q}_{T'}\subseteq \bigcup_{T'<_4 T} \mathcal{Q}_{T'}$.
	
	Now we show that if $T$ is not the first tree according to $<_4$, then $\mathcal{Q}_T$ is not disjoint from $\bigcup_{T'<_4 T}\mathcal{Q}_{T'}$. For this, it is enough to show that there is a tail-edge $e\in T$, as we already proved that in this case $\mathcal{Q}_{T-e}\subseteq \bigcup_{T'<_4 T}\mathcal{Q}_{T'}$. 
	Let $T'$ be any tree preceding $T$ in $<_4$. Let us take the first difference between $T'$ and $T$. If the first difference is an edge that is included into $T$ from the tail direction, then we are ready. Otherwise, since $T'<_4 T$, there are three possibilities :
	
	Case 1: An edge $e$ is included into $T'$ from the head direction, and $e$ is not present in $G$.	
	Case 2: An edge $e$ is included into $T'$ from the head direction, and $e$ is seen from its tail in the tour of $T$, but not included into $T$. 
	Case 3: We see an edge $e$ in the tour of $T'$ that is not present in $G'$, and $e$ is seen from its tail in the tour of $T$, but $e\notin T$.
	
	In cases 1 and 3, we can use Lemma \ref{l:edge_in_or_out_in_two_facets} to deduce that there is a cut $(V_0,V_1)$ such that $e$ is in the cut, moreover, in $G$ each edge of the cut is either not present or oriented from $V_0$ to $V_1$, and in $G'$ each edge of the cut is either not present or oriented from $V_1$ to $V_0$. In both cases we see that $T$ and $T'$ need to differ when we first reach an edge of the cut, since no edge of the cut occurs both in $G$ and $G'$ with the same orientation. Hence in both cases, $e$ is the edge that we first reach from the cut. Hence we conclude that in both cases, $b_0\in V_0$.
	As $T$ is a tree, it needs to contain an edge from the cut $G(V_0,V_1)$. However, as we see, each edge of $G(V_0,V_1)$ points from $V_0$ to $V_1$, hence the first edge of $T$ from the cut is necessarily reached from its tail. This finishes the proof for Cases 1 and 3.
	
	In case 2, we can repeat essentially the same argument using Lemma \ref{l:edge_ordered_differently_in_two_facets}.
	
	
	Now we show that any point $\mathbf{p}\in \mathcal{Q}_T \cap (\bigcup_{T'<_4 T} \mathcal{Q}_{T'})$ is in $\mathcal{Q}_{T-e}$ for some tail-edge $e\in T$. This will finish the proof.
	Suppose that $\mathbf{p}\in \mathcal{Q}_T\cap \mathcal{Q}_{T'}$ where $T'<_4 T$. 
	Let $G$ and $G'$ be the facet graphs of $T$ and $T'$, respectively. Write up $\mathbf{p}$ as a linear combination $\sum_{\overrightarrow{e}\in T\cap T'} \lambda_{\overrightarrow{e}} \mathbf{x}_{\overrightarrow{e}}$.
	Let $H\subset T \cap T'$ be the set of edges of $T$ that are taken with a positive coefficient in the above combination. It is enough to find an edge of $T-H$ that is reached at its tail in the tour of $T$, since in this case $\mathbf{p}\in_{T-e}$. 
	Examine the first difference between the tours of $T$ and $T'$.
	The edges of $H$ are part of $T'$, and with the same orientation as in $G$.
	Hence the first difference between $T$ and $T'$ is an edge outside of $H$. If the first difference is a tail-edge of $T$, then we are ready. If not, then there are three cases:
	
	Case 1: An edge $e$ is included into $T'$ from the head direction, and $e$ is not present in $G$.	
	Case 2: An edge $e$ is included into $T'$ from the head direction, and $e$ is seen from its tail in the tour of $T$, but not included into $T$. 
	Case 3: We see an edge $e$ in the tour of $T'$ that is not present in $G'$, and $e$ is seen from its tail in the tour of $T$, but $e\notin T$.
	
	In cases 1 and 3, we can use Lemma \ref{l:edge_in_or_out_in_two_facets} to deduce that there is a cut $(V_0,V_1)$ such that in $G$ each edge of the cut is either not present or oriented from $V_0$ to $V_1$, and in $G'$ each edge of the cut is either not present or oriented from $V_1$ to $V_0$. In both cases we see that $T$ and $T'$ need to differ when we first reach an edge of the cut, since no edge of the cut occurs both in $G$ and $G'$ with the same orientation. Hence in both cases, $e$ is the edge that we first reach from the cut. Hence we conclude that in both cases, $b_0\in V_0$. We also see that no edge of the cut is in $H$.
	
	As $T$ is a tree, it needs to contain an edge from the cut $G(V_0,V_1)$. However, as we see, each edge of $G(V_0,V_1)$ points from $V_0$ to $V_1$, hence the first edge of $T$ from the cut is necessarily reached from its tail. This finishes the proof for Cases 1 and 3.
	
	In case 2, we can repeat essentially the same argument using Lemma \ref{l:edge_ordered_differently_in_two_facets}.
\end{proof}

\section{An interpretation of $(\gamma_\cG)_1$}
\label{sec:gamma_1}

Ohsugi and Tsuchiya conjecture that the $h^*$-polynomial of the symmetric edge polytope is $\gamma$-positive \cite[Conjecture 4.11]{OT}. 
This has been proved for some graph classes.
For some graph classes (cycles \cite{OT}, complete bipartite graphs \cite{arithm_symedgepoly}, complete graphs \cite{Root_poly_complete_graph}), the $h^*$-polynomial was explicitly computed. Also, Ohsugi and Tsuchiya proved that the $h^*$-polynomial of the symmetric edge polytope is $\gamma$-positive if $G$ has a vertex that is connected to every other vertex \cite[Theorem 5.3]{OT} or if $G$ is bipartite where each partite class contains a vertex that is connected to each vertex of the other partite class \cite[Corollary 5.5]{OT}.

In this section we prove that $\gamma_1$ is nonnegative for any graph, moreover, it equals twice the cyclomatic number. Independently from us, D'Al\`i et al.\ \cite{dalietal} found the same result with a simple proof. Moreover, they also establish the nonnegativity of $\gamma_2$.
(Note that $\gamma_0=1$ is trivially true for each graph, since the constant term of the $h^*$-vector is known to be 1.)

\begin{thm}\label{thm:gamma(1)=2g}
    For any simple undirected graph $\cG$, $$\gamma_1=2g,$$ where $g=|E(\cG)|-|V(\cG)|+1$.
\end{thm}

Let us fix a ribbon structure and a basis for $\cG$, and let $\jaeg(\cG)$ be the union of the Jaeger trees of the facet graphs of $\cG$ for this ribbon structure and basis.
Let $\jaeg_1(\cG)$ denote the Jaeger trees of $\cG$ with exactly one tail-edge. By Corollary \ref{cor:h^*_of_sym_edge_poly}, $h^*_1=|\jaeg_1(\cG)|$. 
We know that $\gamma_0 = 1$ since $h^*_0=1$.
By \cite{OT}, the degree of $h^*_{P_\cG}$ is equal to $|V|-1$. (Note that this also follows from the fact that the last simplex of the shelling of the boundary will glue on all of its $|V|-1$ facets.)
Hence $h^*_1=\gamma_1+(|V|-1)\gamma_0$. Now for proving Theorem \ref{thm:gamma(1)=2g}, it is enough to show that $h^*_1 = |\jaeg_1(\cG)| =  2|E|-|V|+1$. 

We call a Jaeger tree $T$ an $\overrightarrow{e}$-stick Jaeger tree if $\overrightarrow{e}\in T$ is the unique tail-edge of $T$. (The name is motivated by the fact that in this case $\mathcal{Q}_T$ sticks to the previous simplices at the face $\mathcal{Q}_{T-e}$.) We will do a bit more than just counting Jaeger trees with exactly one tail-edge, we will also be able to tell what are these tail-edges.
Let $G_1$ again denote the facet graph of $\cG$ with maximal $f$-value. (That is, the facet graph defined by $l(v)=dist(b_0,v)$.)
Also, let $T_1$ be the first Jaeger tree of $G_1$ with respect to $\prec$ (which is also the first Jaeger tree in $\jaeg(\cG)$ with respect to both $<_f$ and $<_4$).
The following lemma is the heart of the proof of Theorem \ref{thm:gamma(1)=2g}. 
\begin{lemma}\label{l:one-stick_trees}
	For each $\overrightarrow{e}\in T_1$, there is no $\overrightarrow{e}$-stick Jaeger tree in $\jaeg(\cG)$, and there is exactly one $\overleftarrow{e}$-stick Jaeger tree in $\jaeg(\cG)$.
	
	For each edge $e$ of $\cG$ outside of $T_1$, there is exactly one $\overrightarrow{e}$-stick Jaeger tree and exactly one $\overleftarrow{e}$-stick Jaeger tree in $\jaeg(\cG)$.
\end{lemma}

We start by showing that for any oriented edge $\overrightarrow{e}$, there is at most one $\overrightarrow{e}$-stick Jaeger tree. This is in fact true in greater generality, hence we state it in this more general form.

Let $\{\overrightarrow{e}_1, \dots \overrightarrow{e}_r\}$ be a set of oriented edges. Then we call a Jaeger tree $T\in \jaeg(\cG)$ an $\{\overrightarrow{e}_1, \dots \overrightarrow{e}_r\}$-stick Jaeger tree, if $\{\overrightarrow{e}_1, \dots \overrightarrow{e}_r\}$ is exactly the set of tail-edges of $T$ (that is, all these edges are in $T$ and they are first reached at their tail in the tour of $T$, moreover, all other edges of $T$ are first reached at their head).

\begin{lemma}\label{l:at_most_one_sticky_Jaeger_tree}
	For any (possibly empty) set of oriented edges $\{\overrightarrow{e}_1, \dots \overrightarrow{e}_r\}$, there is at most one $\{\overrightarrow{e}_1, \dots \overrightarrow{e}_r\}$-stick Jaeger tree in $\jaeg(\cG)$.
\end{lemma}

We can further strengthen the above lemma in the following way. We will prove this stronger version. 
\begin{lemma}\label{l:one-stick_stronger}
	Let $\{\overrightarrow{e}_1, \dots \overrightarrow{e}_r\}$ be a set of oriented edges.
	Suppose that there is a Jaeger tree $T$ such that $\{\overrightarrow{e}_1, \dots \overrightarrow{e}_r\}\subseteq T$, and $ \{\overrightarrow{e}_1, \dots \overrightarrow{e}_r\}$ contains all the tail-edges of $T$ (and possibly some head-edges). Then there is no $\{\overrightarrow{e}_1, \dots \overrightarrow{e}_r\}$-stick Jaeger tree $T'$, except possibly for $T$. 
\end{lemma}
\begin{proof}
	We can suppose that $\{\overrightarrow{e}_1, \dots \overrightarrow{e}_i\}$ are tail-edges and $\{\overrightarrow{e}_{i+1}, \dots \overrightarrow{e}_r\}$ are head edges in $T$.
	
	
	Suppose for a contradiction that there is a $\{\overrightarrow{e}_1, \dots \overrightarrow{e}_r\}$-stick Jaeger tree $T'$. 
	
	If $T$ and $T'$ are Jaeger trees of the same facet graph $G$, then the first difference in their tours needs to be that an edge $f$ is included into one of them, and not included into the other. As both graphs live in $G$, the edge $f$ is oriented the same way in the two tours. As $T$ and $T'$ are both Jaeger, $\overrightarrow{f}$ needs to be reached at its tail, hence $f$ is a tail-edge in either $T$ or $T'$. However, as $\{\overrightarrow{e}_1, \dots \overrightarrow{e}_r\} \subseteq T\cap T'$, $\overrightarrow{f}\notin \{\overrightarrow{e}_1, \dots \overrightarrow{e}_r\}$.  This is a contradiction, since $\{\overrightarrow{e}_1, \dots \overrightarrow{e}_r\}$ should contain all tail-edges of both $T$ and $T'$.
	
	Now suppose that the facet graph of $T$ is $G$ and the facet graph of $T'$ is $G'\neq G$.
	Suppose that $f(G)>f(G')$. Then by Lemma \ref{l:nonfirst_orientation_exits_dircut} there exist a cut $(V_0, V_1)$ such that $b_0\in V_0$ and in $G(V_0,V_1)$ each edge is either not present or points from $V_1$ to $V_0$ while in $G'(V_0,V_1)$ each edge is either not present or points from $V_0$ to $V_1$.
	As the edges $\{\overrightarrow{e}_1, \dots \overrightarrow{e}_r\}$ are present both in $T$ and in $T'$ with the same orientation, they are present in both $G$ and $G'$ with this orientation, hence none of them is in the cut $(V_0,V_1)$.
	
	As $T'$ is connected, it needs to contain at least one edge from the cut $G'(V_0,V_1)$. Moreover, the edge $\overrightarrow{f}$ of $T'$ that is first reached from $G'(V_0,V_1)$ will be reached at its tail. 
	As by the above reasoning $\overrightarrow{f}\notin \{\overrightarrow{e}_1, \dots \overrightarrow{e}_r\}$, $T'$ cannot be an $\{\overrightarrow{e}_1, \dots \overrightarrow{e}_r\}$-stick Jaeger tree.
	
	If $f(G')>f(G)$, then by the same argument, we conclude that $T$ needs to contain a tail-edge $\overrightarrow{f}\notin \{\overrightarrow{e}_1, \dots \overrightarrow{e}_r\}$, which is again a contradiction.
\end{proof}

Lemma \ref{l:one-stick_stronger} implies that in all the cases listed in \ref{l:one-stick_trees}, there is at most one $\overrightarrow{e}$-stick Jaeger tree.
Moreover, it also implies that there is no $\overrightarrow{e}$-stick Jaeger tree for $\overrightarrow{e}\in T_1$. Indeed $\overleftarrow{e}\in T_1$ is a head-edge of $T_1$ (since $T_1$ has only head-edges). Hence we can apply Lemma \ref{l:one-stick_stronger} to $\{\overrightarrow{e}\}$.

However, to characterize Jaeger trees in $\jaeg_1(\cG)$ we also need ways to construct Jaeger trees with one tail-edge.
For this, we will use the following greedy procedure. Intuitively, we traverse the graph and build up a tree while obeying the Jaeger rule (not cutting edges at their head), and always cutting edges at their tail if they are not already in the tree. We do not consider hidden edges in this process.

\begin{defn}[greedy tree]\label{def:greedy_tree}
	Let $G$ be a ribbon digraph with a basis $(b_0,b_0b_1)$.
	Let us call the outcome of the following procedure a greedy tree. We start with the empty subgraph $H=\emptyset$, and with current vertex $b_0$ and current edge $b_0b_1$. At any moment, if the current node-edge pair is $(h, ht)$ where $h$ is the head of $ht$, then if $(t,th)$ has not yet been current node-edge pair, then we include $ht$ into $H$, traverse $ht$ and take $(t, th^+)$ as the next current node-edge pair. If $(t,th)$ has already been current node edge pair, then we do not include $ht$ into $H$, an take $(h, ht^+)$ as next current node-edge pair. 
	If at some moment the current node-edge pair is $(t, th)$ where $t$ is the tail of $th$, and $(h, ht)$ has not yet been current node-edge pair, then do not include $th$ into $H$ and take $(t, th^+)$ as the next current node-edge pair. If $(h, ht)$ has already been current node-edge pair, then take $(h, ht^+)$ as next current node-edge pair. The process stops when a node-edge pair becomes current for the second time. The output is the subgraph $H$.
\end{defn}
See the left panel of Figure \ref{fig:greedy_and_almost_greedy_trees} for an example.

\begin{lemma}\label{l:greedy_tree_is_a_tree}
	The above process produces a (not necessarily spanning) tree.
\end{lemma}
\begin{proof}
	Suppose for a contradiction that a cycle appears in $H$, and suppose that this first happens when an edge $\overrightarrow{xy}$ is included into $H$. Stop the process immediately before the inclusion of $\overrightarrow{xy}$. Then the current node-edge pair is $(y,yx)$. Let $C$ be the unique cycle that we get upon including $\overrightarrow{xy}$ into $H$. Let us call its vertices $y=v_0, v_1, \dots v_k= x$. 
	As $C$ is semi-balanced, half the edges need to be oriented in one cyclic direction, and half of them in the other one. Hence we need to have at least one edge $v_jv_{j+1}$ that is oriented as $\overrightarrow{v_{j+1}v_{j}}$ (as $\overrightarrow{xy}$ stands in the opposite direction to this).
	
	Notice that until now, during the process, around any vertex $v$, the node-edge pairs $\{(v, vu): uv \in E\}$ became current in an order compatible with the ribbon structure. Hence if a vertex $v$ is first reached through an edge $uv$, then $(v,vu)$ becomes current only after each $(v, vu')$ for $u'\in\Gamma(v)-u$ has become current (where $\Gamma(v)$ denotes all in-, and out-neighbors of $v$).
	
	Notice also, that during the process, we always move along the  (currently present) edges of $H$. Hence (since we have been in $x$, and then we later get to $y$) we must have traversed each edge $v_iv_{i+1}$ from $v_{i+1}$ to $v_{i}$. That is, each node-edge pair $(v_{i+1},v_{i+1}v_{i})$ (for $i=0, \dots k-1$) has been current before $(y,yx)$.
	
	Take the edge $v_jv_{j+1}$ which is oriented as $\overrightarrow{v_{j+1}v_{j}}$. As $\overrightarrow{v_{j+1}v_j}\in H$, $(v_{j},v_{j}v_{j+1})$ was current before $(v_{j+1},v_{j+1},v_j)$. Hence $v_{j+1}$ was first reached through $\overleftarrow{v_jv_{j+1}}$. Thus, by our previous note, each node-edge pair $(v_{j+1},v_{j+1}u)$ for $u\neq v_j$ became current before $(v_{j+1},v_{j+1},v_j)$. In particular, $(v_{j+1},v_{j+1}v_{j+2})$ also became current before $(v_{j+1},v_{j+1},v_j)$. As we have not traversed a cycle by that point, this must be the first time when $v_{j+1}v_{j+2}$ became current. Hence (since $v_{j+1}v_{j+2} \in H$), we need to have the orientation $ \overleftarrow{v_{j+1}v_{j+2}}$.
	
	Continuing this way, we get that $v_{k-1}x$ also has to be oriented as $\overleftarrow{v_{k-1}x}$. But then $(x,\overrightarrow{xy})$ has to become current before $(x,\overrightarrow{xv_{k-1}})$, and hence before $(y,\overleftarrow{yx})$. But this implies that we will not include $\overrightarrow{xy}$ into $H$, a contradiction.
	
\end{proof}

Note the following:
\begin{claim}
If the greedy tree of a digraph $G$ is a spanning tree, then it is a Jaeger tree.
\end{claim}
\begin{proof}
    This follows from the construction, as the tour of the greedy tree agrees with the walk as we build it up, and we always include any edge that is first reached at its head.
\end{proof}

The greedy tree might not be spanning. Indeed, if there is a cut $(V_0,V_1)$ in the graph $G$ such that $b_0\in V_0$ and all edges of $G(V_0,V_1)$ point toward $V_1$, then the process will never get to $V_1$. However, this is the only obstacle that can prevent the greedy tree from being spanning. Indeed, suppose that the greedy tree is not spanning for a digraph, and let $V_0$ be the vertex set of the greedy tree, and let $V_1$ be the rest of the vertices. We claim that each edge of the cut $(V_0,V_1)$ is oriented toward $V_1$. At any node of $V_0$, all incident edges eventually become current. Hence if there were any edges in the cut $(V_0,v_1)$ oriented toward $V_0$, then the first reached such edge would be included into $H$, a contradiction.

Hence in the facet graph $G_1$ (that does not contain a bad cut by Lemma \ref{l:reversing_cut_pointing_away}), the greedy tree will be a Jaeger tree. 
It is also clear by its construction that it will not contain any tail-edge.
We have seen in the proofs of Theorems \ref{thm:shelling_of_symm_edge_poly} and \ref{t:quadratic_shelling} that for both $<_f$ and $<_4$ if a Jaeger tree is not the first one, then it contains at least one tail-edge. Hence the greedy tree of $G_1$ is the first Jaeger tree according to both $<_f$ and $<_4$, the one that we were referring to as $T_1$. 

To understand $\jaeg_1(\cG)$, we will also need a slightly modified version of the greedy tree construction.

\begin{defn}[almost-greedy tree]
	Let $G$ be a ribbon digraph, $(b_0,b_0b_1)$ a basis, and $\overrightarrow{uv}$ an arbitrary edge. 
	
	Let us call the outcome of the following procedure an $\overrightarrow{uv}$-almost greedy tree. We follow the steps of Definition \ref{def:greedy_tree}, with the exception that if $(u,uv)$ becomes current node-edge pair before $(v, vu)$, then (when it becomes current) we include $uv$ into $H$, and take $(v,vu^+)$ as next current node-edge pair.
\end{defn}
See the rightmost panel of Figure \ref{fig:greedy_and_almost_greedy_trees} for an example of an almost greedy tree.

\begin{lemma}\label{l:almost_greedy_tree}
	The above process produces a (not necessarily spanning) tree.
\end{lemma}
\begin{proof}
	The proof of Lemma \ref{l:greedy_tree_is_a_tree} carries over with a few modifications. Once again suppose for a contradiction that a cycle $C$ appears, and stop the process when it first happens. Suppose again that in this moment $(y,yx)$ is the current edge (that is, we include $xy$ into $H$ and this creates the cycle $C$). Once again denote the vertices of $C$ by $y=v_0, v_1, \dots v_k= x$.
	
	As $G$ does not contain multiple edges, $|C|\geq 4$. If $xy \neq uv$ then $xy$ needs to have orientation $\overrightarrow{xy}$, as we put it into $H$ when $y$ is the current node.
	There must be at least 2 edges in $C$ standing opposite to $\overrightarrow{xy}$, and hence these edges are oriented as $\overrightarrow{v_{i+1}v_{i}}$. One of these might by $uv$, but at least one of them has to be a ``regular'' edge. For that edge, we can repeat the argument of the proof of Lemma \ref{l:greedy_tree_is_a_tree}, and obtain that $(x,xy)$ was current before $(y,xy)$, which is a contradiction, as in such a case we would not put $xy$ into $H$.
	
	If $xy = uv$, then it might be oriented as $\overrightarrow{yx}$. As $|C|\geq 4$, we still need to have some other edge $v_iv_{i+1}$ oriented as $\overrightarrow{v_{i+1}v_{i}}$. In this case, we also conclude that $(x,xy)$ was current before $(y,xy)$ which is still a contradiction, since in this case we would already have included $xy$ into $H$ at its endpoint $x$.
\end{proof}

    \begin{figure}
    	\begin{center}    
    		\begin{tikzpicture}[-,>=stealth',auto,scale=0.5,thick]
    		\tikzstyle{o}=[circle,scale=0.6,fill,draw]
    		\node[o, label=left:$y$] (0) at (0, 0) {};
    		\node[o, label=left:$x$] (1) at (-2, 2) {};
    		\node[o] (2) at (0, 2) {};
    		\node[o, label=right:$a$] (3) at (2, 2) {};
    		\node[o] (4) at (-1, 4) {};
    		\node[o] (5) at (1, 4) {};
    		\node[o] (6) at (3, 4) {};
    		\node[o] (7) at (2, 6) {};
    		\path[->,line width=0.7mm]
    		(1) edge node [above] {} (0)
    		(2) edge node [below] {} (0)
    		(3) edge node [above] {} (0)
    		(4) edge node [below] {} (2)
    		(5) edge node [below] {} (3)
    		(6) edge node [below] {} (3)
    		(7) edge node [below] {} (6);
    		\path[->,line width=0.2mm]
    		(4) edge node [below] {} (1)
    		(5) edge node [below] {} (2)
    		(7) edge node [below] {} (5);
    		\path[-,dotted]
    		(4) edge node [below] {} (5);
    		\end{tikzpicture}
    		\hspace{0.3cm}
    		\begin{tikzpicture}[-,>=stealth',auto,scale=0.5,thick]
    		\tikzstyle{u}=[circle]
    		\tikzstyle{o}=[circle,scale=0.6,fill,draw]
	        \draw[rotate=-25] (-2.75, 2.2) ellipse (1.1 and 2.3);
	        \draw[rotate=-25] (-0.65, 5.2) ellipse (1 and 2.1);
    		\node[u] (e) at (-1.2, 5.8) {$S_x^d$};
    		\node[u] (e) at (3.6, 6.5) {$S_x^h$};
    		\node[o, label=left:$y$] (0) at (0, 0) {};
    		\node[o, label=left:$x$] (1) at (-2, 2) {};
    		\node[o
    		] (2) at (0, 2) {};
    		\node[o, label=right:$a$
    		] (3) at (2, 2) {};
    		\node[o
    		] (4) at (-1, 4) {};
    		\node[o
    		] (5) at (1, 4) {};
    		\node[o
    		] (6) at (3, 4) {};
    		\node[o
    		] (7) at (2, 6) {};
    		\path[->,every node/.style={font=\sffamily\small}]
    		(1) edge node [above] {} (0)
    		(2) edge node [below] {} (0)
    		(3) edge node [above] {} (0)
    		(4) edge node [below] {} (1)
    		(4) edge node [below] {} (2)
    		(5) edge node [below] {} (2)
    		(5) edge node [below] {} (3)
    		(6) edge node [below] {} (3)
    		(7) edge node [below] {} (5)
    		(7) edge node [below] {} (6);
    		\path[-,dotted,every node/.style={font=\sffamily\small}]
    		(4) edge node [below] {} (5);
    		\node[u] (e) at (-1.4, 0.9) {$e$};
    		\end{tikzpicture}
    		\begin{tikzpicture}[-,>=stealth',auto,scale=0.5,thick]
    		\tikzstyle{u}=[circle]
    		\tikzstyle{o}=[circle,scale=0.6,fill,draw]
    		\node[o, label=left:$y$] (0) at (0, 0) {};
    		\node[o, label=left:$x$] (1) at (-2, 2) {};
    		\node[o] (2) at (0, 2) {};
    		\node[o, label=right:$a$] (3) at (2, 2) {};
    		\node[o] (4) at (-1, 4) {};
    		\node[o] (5) at (1, 4) {};
    		\node[o] (6) at (3, 4) {};
    		\node[o] (7) at (2, 6) {};
    		\path[->,line width=0.2mm,]
    		(2) edge node [below] {} (4);
    		\path[->,line width=0.7mm]
    		(0) edge node [above] {} (1)
    		(2) edge node [below] {} (0)
    		(3) edge node [above] {} (0)
    		(4) edge node [below] {} (1)
    		(5) edge node [below] {} (4)
    		(6) edge node [below] {} (3)
    		(7) edge node [below] {} (5);
    		\path[-,dotted,every node/.style={font=\sffamily\small}]
    		(5) edge node [below] {} (2)
    		(5) edge node [below] {} (3)
    		(7) edge node [below] {} (6);
    		\node[u] (e) at (-1.4, 0.9) {$e$};
    		\end{tikzpicture}
    	\end{center}
\caption{We consider the ribbon structure induced by the positive 
orientation of the plane, with basis $(y,ya)$. \\
Left panel: The first facet graph $G_1$ of the underlying undirected graph (hidden edges 
are dotted), and its greedy tree (thick edges). \\
Middle panel: The 
sets $S_x^d$ and $S_x^h$ for the edge $\protect\overrightarrow{e} = 
\protect\overrightarrow{xy}$, as defined in the proof of Lemma 
\ref{l:one-stick_trees}. \\
Right panel: The orientation obtained for 
$xy$ in Case 3 of the proof of Lemma \ref{l:one-stick_trees}, and 
the $\protect\overrightarrow{yx}$ - almost greedy tree (thick edges). }
\label{fig:greedy_and_almost_greedy_trees}
    \end{figure}
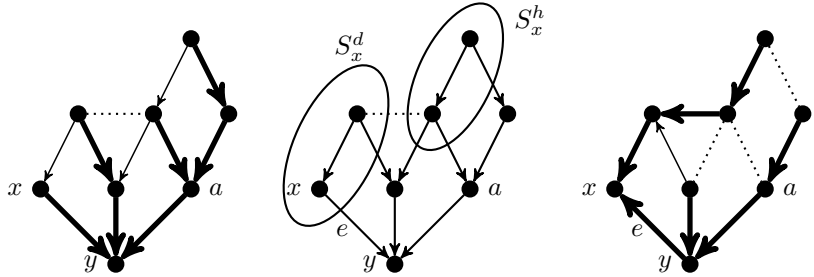 

\begin{proof}[Proof of Lemma \ref{l:one-stick_trees}]
	\textbf{Case 1:} Take an edge $\overrightarrow{e}\in T_1$. We show that $\overrightarrow{e}$ cannot be the unique tail-edge in a Jaeger tree (of and facet graph). Indeed, as $T_1$ has no tail-edges, Lemma \ref{l:one-stick_stronger} applied to $\{\overrightarrow{e}\}$ tells us that there is no $\overrightarrow{e}$-stick Jaeger tree.
	
	

    \textbf{Case 2:} Let $\overrightarrow{e}\in G_1-T_1$. 
    
    Lemma \ref{l:at_most_one_sticky_Jaeger_tree} tells us that we cannot have more than one $\overrightarrow{e}$-stick Jaeger tree. 
    Take the $\overrightarrow{e}$-almost greedy tree in $G_1$, and call it $T$. We show that $T$ is an $\overrightarrow{e}$-stick Jaeger tree. 
    
    By Lemma \ref{l:almost_greedy_tree}, this is a tree. Also, we claim that $\overrightarrow{e}\in T$ and $\overrightarrow{e}$ is a tail-edge of $T$. Indeed, since $T_1$ is the greedy tree of $G_1$, and $\overrightarrow{e}\notin T_1$, 
    in the greedy tree process, at some point we reach $\overrightarrow{e}$ at its tail (and do not include it to $T_1$) before seeing $\overrightarrow{e}$ from its head. Hence in the $\overrightarrow{e}$-almost greedy tree in $G_1$, $\overrightarrow{e}$ will be included so that it is reached at its tail. We also claim that $T$ is spanning. Indeed, the only way for not reaching some vertices would be if they are separated from $b_0$ by a cut directed away from the part of $b_0$, but there is no such cut in $G_1$.
    
    Notice that the tour of $T$ agrees with the way we traversed the graph when building up $T$. In that process, we traversed each edge that was first seen from the head direction. Hence $T$ is a Jaeger tree. Moreover, by construction, $\overrightarrow{e}$ is the only tail-edge, hence $T$ is indeed an $\overrightarrow{e}$-stick Jaeger tree. 
    
    
    
    
    \textbf{Case 3:}
    Now we show that for an edge $\overrightarrow{e}\in G_1$, (no matter whether $\overrightarrow{e}$ is in $T_1$ or not) there is exactly one $\overleftarrow{e}$-stick Jaeger tree for $\cG$. Lemma \ref{l:at_most_one_sticky_Jaeger_tree} tels us that there is at most one.
    To show that there is at least one $\overleftarrow{e}$-stick Jaeger tree, we need to find out which orientation should contain this Jaeger tree. To get this orientation, take the following layering:
    Let $l$ be the layering of $G_1$ (that is, $l(v)$ is the distance in $\cG$ from $b_0$). We would like to modify this so that $\overrightarrow{e}=\overrightarrow{xy}$ is reversed. Take $l'(v)=l(v)$ for $v\neq x$ and $l'(x)=l(x)+2$. This will cause $xy$ to get reversed, but it might not be an appropriate layer function. Indeed, if there was any edge $\overrightarrow{wx}$ pointing to $x$ in $G_1$, then now $l'(x)-l'(w)=3$. Also, if there was a hidden edge $zx$ incident to $x$ then we have $l'(x)-l'(z)=2$. Hence we need to increase $l'$ by 2 on the in-neighbors of $x$, and on the in-neighbors of those, etc. And we need to increase $l'$ by 1 on the vertices that are connected by a hidden edge to a node on which $l$ was increased by 2, and also on their in-neighbors, etc. 
    More formally, let $$S^d_x=\{v\in V: \text{ $x$ is reachable on a directed path from $v$ in $G_1$}\},$$ and let 
    \begin{align*}
    S^h_x=\{v\in V: \exists u \in V \text{ such that $u$ is reachable from $v$ on a directed}\\ \text{ path in $G_1$ and $u$ is connected to $S^d_x$ by a hidden edge}\}.
    \end{align*}
    For an example, see the middle panel of Figure \ref{fig:greedy_and_almost_greedy_trees}. Then let 
    $$
    l''(v) = \left\{\begin{array}{cl} 
    l(v)+2 & \text{if $v\in S^d_x$}  \\
    l(v)+1  & \text{if $v \in S^h_x$}\\
    l(v)  & \text{otherwise}
    \end{array} \right..
    $$
    It is easy to check that $l''$ is a good layering, that is, $|l''(u)-l''(v)|\leq 1$ for each edge $uv \in E(\cG)$. 
    Also, we claim that $l''(y)=l(y)$, thus, $xy$ indeed gets reversed. Indeed, each vertex in $v\in S^d_x$ has $l(v)<l(y)$, and each vertex in $u\in S^h_x$ has $l(u)\leq l(v)$ for some $v\in S^d_x$, hence $y\notin S^d_x\cup S^h_x$.
    
    Hence $l''$ defines a semi-balanced subgraph $G$ in which $xy$ is oriented as $\overrightarrow{yx}$. (For an example, see the right panel of Figure \ref{fig:greedy_and_almost_greedy_trees}.) We show that $G$ has a Jaeger tree where $\overrightarrow{yx}$ is the only tail-edge. For this, let us take the $\overrightarrow{yx}$-almost greedy tree $T$ in $G$. We need to show that $T$ is spanning. Then by construction it also follows that it is a Jaeger tree. Also, we need to show that $T$ indeed contains $\overrightarrow{yx}$, reached at its tail.
    
    Notice that $(V-(S^d_x\cup S^h_x), S^d_x\cup S^h_x)$ defines a directed cut in $G$: all edges of $\cG(V-(S^d_x\cup S^h_x), S^d_x\cup S^h_x)$ are either not present, or point from $V-(S^d_x \cup S^h_x$ to $S^d_x \cup S^h_x$. 
    In $G$, all edges of $\cG$ between $S^h_x$ and $S^d_x$ point toward $S^d_x$. As by the definition of $S^h_x$, each vertex of $S^h_x$ is connected to some vertex of $S^d_x$, we conclude that in $G$, $x$ is reachable from each vertex of $S^d_x\cup S^h_x$.
    
    Also, $b_0\in V-(S^d_x\cup S^h_x)$ as $b_0$ is a sink in $G_1$. Hence the almost greedy tree process will start at $V-(S^d_x\cup S^h_x)$ and as it can only get through the cut $G(V-(S^d_x\cup S^h_x),S^d_x\cup S^h_x)$ via $\overrightarrow{yx}$, it will eventually reach $\overrightarrow{yx}$ at $y$. Thus, it also traverses $\overrightarrow{yx}$. Now from each vertex, either $x$ or $b_0$ is reachable in $G$. Hence $T$ spans $G$.
    We conclude that $T$ is a Jaeger tree of $G$, and by construction $\overrightarrow{yx}$ is the only tail-edge of $T$.
    
    \textbf{Case 4:}
    Let $e\in\cG$ be an edge that is not present in $G_1$. We show that there is exactly one $\overrightarrow{e}$-stick Jaeger tree, and exactly one $\overleftarrow{e}$-stick Jaeger tree. Once again, Lemma \ref{l:at_most_one_sticky_Jaeger_tree} implies the ``at most one'' part.

    Let $e=xy$. The fact that $xy$ is hidden in $G_1$ means that $l(x)=l(y)$.
    By symmetry, it is enough to show that there is a $\overrightarrow{yx}$-stick Jaeger tree. For this, take $l'(x)=l(x)+1$, while $l'(v)=l(v)$ for $v\in V-x$. This might not be a good layering, because if $x$ has any in-edge $\overrightarrow{wx}$, then $l'(x)-l'(v)=2$. Hence we need to increase $l'$ by one on the vertices from where $x$ is reachable on a directed path. Let $S_x=\{v\in V: \text{ $x$ is reachable on a directed path from $v$ in $G_1$}\}$, and let $l''(v)=l(v)+1$ for $v\in S_x$ and $l''(v)=l(v)$ for $v\notin S_x$. Then $l''$ is a good layering, moreover, for the obtained facet graph $G$, $G(V-S_x, S_x)$ is a cut where each edge is either not present or oriented from $V-S_x$ to $S_x$. Now we can repeat the proof of the previous case.
\end{proof}

\begin{proof}[Proof of Theorem \ref{thm:gamma(1)=2g}]
    As we argued after the statement of Theorem \ref{thm:gamma(1)=2g}, it is enough to prove
	 $|\jaeg_1(\cG)|=2|E|-|V|+1$. This follows from Lemma \ref{l:one-stick_trees}, hence we are ready.
\end{proof}

\begin{problem}
	Give a combinatorial meaning for $\gamma(i)$ for larger values of $i$.
\end{problem}

\section{On the connection of $\gamma$-polynomials and interior polynomials}
\label{sec:connection_of_gamma_and_interior}

Let $\cG$ be an undirected graph, and $G$ one of its facet graphs. In \cite{semibalanced}, the $h^*$-vector of the root polytope of $G$ (which is a facet of $P_\cG$) is called the interior polynomial of $G$.

If $\cG$ is a bipartite graph, then there are two special facet graphs, namely, the standard orientations of $\cG$, where each edge is oriented from partite class $U$ to partite class $W$, or vice versa. The root polytopes of these two orientations are mirror images of each other, hence the interior polynomials of the two standard orientations are the same. We call this polynomial the interior polynomial of $\cG$, and denote it by $I_\cG$. (We note that in fact interior polynomials were originally defined for undirected bipartite graphs, and the original definition is equivalent to the one given here.)

In this section, we gather results that suggest interesting connections between $\gamma_\cG$ and the interior polynomial $I_\cG$, and also formulate some conjectures.

First of all, for a special class of graphs, Ohsugi and Tsuchiya proved that $\gamma_\cG$ can be written as a linear combination of interior polynomials of come cuts in $\cG$.

For a bipartite graph $\cG$, they define $\tilde{\cG}$ as the bipartite graph where a new vertex is added to both partite classes, and the new vertices are connected to each vertex of the other partite class (including the other new vertex). For a graph $\cG$, $Cut(G)$ is the set of cuts of $\cG$, taken as subgraphs ($|Cut(\cG)|=2^{|V|-1}$, each $S\subset V$ determines the cut $(S,V-S)$ with $S$ and $V-S$ determining the same cut).
They show 

\begin{thm}\cite[Corollary 5.5]{OT}
    For a bipartite graph $\cG$, 
    $$\gamma_{\tilde{\cG}}(x) = \frac{1}{2^{|V|-2}}\sum_{\mathcal{H}\in Cut(\cG)} I_{\tilde{\mathcal{H}}}(4x).$$
\end{thm}

This means that for bipartite graphs where each partite class has a vertex that is connected to each vertex of the other partite class, the gamma polynomial can be written as a positive linear combination of interior polynomials of some subgraphs. As each coordinate of an interior polynomial is nonnegative, this implies $\gamma$-positivity for these graphs, as \cite{OT} points out. (We note that in \cite[Theorem 5.3]{OT} they give another similar formula for (not necessarily bipartite) graphs where some vertex is connected to any other vertex.)

It would be interesting to see if a similar formula is true for all bipartite graphs. Here we list some more results suggesting that the two polynomials should be strongly related in general, and we formulate some conjectures.

One thing that suggests a connection is that certain product formulas are true for both polynomials.
Ohsugi and Tsuchiya proved the following product formula for $\gamma$-polynomials: 
\begin{thm}\cite[Corollary 5.5]{OT}
Let $\mathcal{G}_1$ and $\mathcal{G}_2$ be bipartite graphs with $V(\mathcal{G}_1)\cap V(\mathcal{G}_2)=\{u,w\}$, so that $E(\mathcal{G}_1)\cap E(\mathcal{G}_2)=\{uw\}$. Let $\mathcal{G}_1\cup \mathcal{G}_2$ be the bipartite graph glued along the edge $uw$.
Then $\gamma_{\mathcal{G}_1 \cup \mathcal{G}_2}=\gamma_{\mathcal{G}_1}\cdot \gamma_{\mathcal{G}_2}$.
\end{thm}

The analogous statement holds for interior polynomials. 
\begin{thm}\cite[Theorem 6.7]{hiperTutte}
Let $\mathcal{G}_1$ and $\mathcal{G}_2$ be bipartite graphs with $V(\mathcal{G}_1)\cap V(\mathcal{G}_2)=\{u,w\}$, so that $E(\mathcal{G}_1)\cap E(\mathcal{G}_2)=\{uw\}$. Let $\mathcal{G}_1\cup \mathcal{G}_2$ be the bipartite graph glued along the edge $uw$.
Then $I_{\mathcal{G}_1 \cup \mathcal{G}_2}=I_{\mathcal{G}_1}\cdot I_{\mathcal{G}_2}$.
\end{thm}

We note that a verbatim generalization of this statement holds for interior polynomials of all semi-balanced digraphs, which is stated as Proposition 11.1 in \cite{semibalanced}.

For both polynomials, the same product formula holds also if two graphs are glued along one vertex instead of one edge, see \cite[Proposition 5.2]{OT} and \cite[Corollary 6.8]{hiperTutte}. 

Computations also suggest connections between $\gamma_\cG$ and $I_\cG$.
We computed $\gamma$ and interior polynomials of 12000 randomly chosen bipartite graphs with 6 to 11 vertices.
Based on the computations, we formulate the following conjectures:

\begin{conject}
For any bipartite graph $\cG$, the degrees of $I_\cG$ and of $\gamma_\cG$ agree.
\end{conject}

\begin{conject}\label{conj:gamma_determines_interior}
For any bipartite graph $\cG$, $\gamma_\cG$ determines $I_{\cG}$. That is, if for two bipartite graphs, their $\gamma$-polynomial agrees, then their interior polynomial agrees as well.
\end{conject}
We note that the converse of Conjecture \ref{conj:gamma_determines_interior} is not true. In particular, we found two graphs both having $3x^2+3x+1$ as interior polynomial, but one of them having $12x^2+6x+1$ as $\gamma$-polinomial, while the other one having $14x^2+6x+1$ as $\gamma$-polinomial.

\begin{conject}
For any bipartite graph $\cG$,
among the facet graphs of $\cG$, the two standard orientations of $\cG$ have a coefficientwise minimal interior polynomial.
\end{conject}
We note that for non-bipartite graphs there is no sense of standard orientation, and it is not clear how one could generalize the preceding conjectures to non-bipartite graphs.

We can also see some connections between the low-degree coefficients of $I_\cG$ and $\gamma_\cG$. By \cite[Theorem 6.3]{hiperTutte}, the coefficient of the constant term of the interior polynomial $I_\cG$ is the nullity $g(\cG)$. While Theorem \ref{thm:gamma(1)=2g} tells us that $\gamma_1(\cG)=2g(\cG)$. However, for coefficients of larger degree, there is no such simple connection, and the ratio of the quadratic terms can be non-integer.

\section{Some concrete special cases}
\label{sec:concrete_cases}

We would like to demonstrate that the technique of Jaeger trees is also applicable to compute the $h^*$-polynomial of symmetric edge polytopes in concrete cases. Here we treat trees, cycles and cacti, complete bipartite graphs, and planar graphs. For the first four classes, the $h^*$-vector was determined previously using Gröbner bases. 

\subsection{Warmup: Trees} 
As a warmup, let us compute the $\gamma$-polinomial of $P_\cG$ if $\cG$ is a tree.

\begin{prop}\cite[Example 5.1]{OT}
Let $\cG$ be a tree. Then $(\gamma_\cG)_0=1$, and $(\gamma_\cG)_i=0$ for each $i>0$.
\end{prop}
\begin{proof}
    Suppose that $\cG$ has $m$ edges. As $\cG$ is bipartite, the facet graphs of $\cG$ are orientations of $\cG$. As there is no cycle, all $2^m$ orientations of $\cG$ are semibalanced. In each orientation, there is exactly one Jaeger tree, that contains all edges. Moreover, the tour of each of these Jaeger trees is the same, hence the orientation of an edge will determine if it is going to be a head-edge or a tail-edge. Hence there are exactly $\binom{m}{k}$ Jaeger trees in the facet graphs of $\cG$ with exactly $k$ tail-edges. This means that $(\gamma_\cG)_0=1$, and $(\gamma_\cG)_i=0$ for each $i>0$.
\end{proof}

\subsection{Cycles and cacti}

Using Gröbner bases, Ohsugi and Tsuchiya \cite{OT} compute the $\gamma$-polynomial of $P_\cG$ if $\cG$ is a cycle. Here we reproduce this result using our methods.

\begin{thm}\cite[Proposition 5.7]{OT}
	$\gamma_{C_{n}}(i)=\binom{2i}{i}$ for $i=0, 1, \dots, \lfloor\frac{n-1}{2}	\rfloor$.
\end{thm}
Let the vertex set of the cycle be $\{v_0, \dots , v_{n-1}\}$, with edges $v_{i}v_{i+1}$ (modulo $n$) for $i=0, \dots n-1$. Let us call an edge $v_iv_{i+1}$ positively oriented if it points towards $v_{i+1}$ and negatively oriented otherwise. 

The following is the key lemma. 
\begin{lemma}
	Let us fix the base $(v_0, v_0v_1)$, and let $0\leq i\leq \lfloor\frac{n-1}{2}\rfloor$ be fixed. The total number of Jaeger trees with exactly $i$ tail-edges in all the facet graphs of $C_n$ is $\sum_{j=0}^i \binom{2j}{j}\binom{n-1-2j}{i-j}$.
\end{lemma}

\begin{proof} 
	Suppose first that $n=2k+1$ is odd. Then the facet graphs are exactly the spanning trees of $C_n$, with any orientation that has $k$ positively and $k$ negatively oriented edges. Each such facet has one Jaeger tree. For a facet graph where the edge $v_av_{a+1}$ is missing, the number of tail-edges is 
	\begin{multline*}
	\sharp\{\text{positively oriented edges between $v_0$ and $v_a$}\}\\
	+\sharp\{\text{negatively oriented edges between $v_{a+1}$ and $v_0$}\}.
	\end{multline*}
	Suppose that the first term is $0\leq j\leq i\leq k$ and the second is $i-j$. Then between $v_0$ and $v_a$, there need to be $j$ positively oriented edges, and $k-(i-j)\geq 0$ negatively oriented edges, and between $v_{a+1}$ and $v_0$ there need to be $k-j \geq 0$ positively oriented edges and $i-j$ negatively oriented edges. This also implies $a=k-i+2j$. Once we decide the value of $i$ and $j$, all choices of $j$ positively oriented edges between $v_0$ and $v_a$ and $i-j$ negatively oriented edges between $v_{a+1}$ and $v_0$ give us exactly one Jaeger tree with $i$ tail-edges.
	Hence altogether, the number of Jaeger trees in all facet graphs with $i$ tail-edges is 
	$$
	\sum_{j=0}^i \binom{k-i+2j}{j}\binom{k+i-2j}{i-j}=\sum_{j=0}^i \binom{2j}{j}\binom{2k-2j}{i-j}=\sum_{j=0}^i \binom{2j}{j}\binom{n-1-2j}{i-j},
	$$
	where the first equality is by Lemma \ref{l:binom_azonossag}.
	
	Now let us look at the case of $n=2k$. In this case, the facet graphs are orientations of $C_{2k}$ where there are exactly $k$ positively and $k$ negatively oriented edges. It is easy to see that each such facet has exactly $k$ Jaeger trees: one of the $k$ positively oriented edges can be cut. If the (positive) edge between $\overrightarrow{v_av_{a+1}}$ is cut, then this tree has $i$ tail-edges iff there are $0\leq j\leq i$ positive edges between $v_0$ and $v_a$ and $i-j$ negative edges between $v_{a+1}$ and $v_0$. This implies (as $v_av_{a+1}$ was positive) that there are $k-i+j$ negative edges between $v_0$ and $v_a$ and $k-1-j$ positive edges between $v_{a+1}$ and $v_0$ and $a=k-i+j$. 
	Hence altogether, the number of Jaeger trees with $i$ tail-edges in all facet graphs is 
	\begin{multline*}
	\sum_{j=0}^i \binom{k-i+2j}{j}\binom{k-1+i-2j}{i-j}\\
	=\sum_{j=0}^i \binom{2j}{j}\binom{2k-1-2j}{i-j}=\sum_{j=0}^i \binom{2j}{j}\binom{n-1-2j}{i-j},
	\end{multline*}
	where the first equality is by Lemma \ref{l:binom_azonossag}.
\end{proof}

\begin{lemma}\label{l:binom_azonossag}
	For any nonnegative integers $(b,c,n)$ such that $0\leq c\leq b-2n$, we have
	$$
	\sum_{a=0}^n \binom{2a}{a}\binom{b-2a}{n-a} = \sum_{a=0}^n \binom{2a+c}{a}\binom{b-c-2a}{n-a}.
	$$
\end{lemma}
\begin{proof}
	We fix $c$ and proceed by induction on $n$ and $b$. The base case is $b=2n + c$. In that case, notice that by substituting $a'=n-a$ in the right side, we get $$\sum_{a'=0}^n \binom{b-2a'}{n-a'}\binom{2a'}{a'},$$ which agrees with the left side.
	
	Now suppose that $b>2n+c$ and suppose that for smaller $b$, we know the statement for each triple $(b,c,n)$ that satisfies $b\geq 2n + c$.
	
	We define the following two functions on subsets of $[b]$: For $S\subseteq [b]$, let
	$f_0(S)=|\{a\leq b: |S\cap [2a]|=a\}|$ and $f_c(S)=|\{a\leq b: |S\cap [2a+c]|=a\}|$.
	
	It is easy to see that $\sum_{a=0}^n \binom{2a}{a}\binom{b-2a}{n-a}= \sum_{S\subseteq [b], |S|=n} f_0(S)$, as well as that $\sum_{a=0}^n \binom{2a+c}{a}\binom{b-c-2a}{n-a}= \sum_{S\subseteq [b], |S|=n} f_c(S)$.
	
	Now $$\sum_{S\subseteq [b], |S|=n} f_0(S)=\sum_{S\subseteq [b], b\notin S, |S|=n} f_0(S)+\sum_{S\subseteq [b], b\in S, |S|=n} f_0(S),$$
	and similarly,
	$$\sum_{S\subseteq [b], |S|=n} f_c(S)=\sum_{S\subseteq [b], b\notin S, |S|=n} f_c(S)+\sum_{S\subseteq [b], b\in S, |S|=n} f_c(S).$$
	
	Notice that whether an element $i>2n+c$ is in $S$ or not does not change neither the value of $f_0(S)$ nor the value of $f_c(S)$. (In fact, for $f_0(S)$, already the elements $i>2n$ are immaterial.) Hence $\sum_{S\subseteq [b], b\notin S, |S|=n} f_0(S)=\sum_{S\subseteq [b-1], |S|=n} f_0(S)$ and $$\sum_{S\subseteq [b], b\in S, |S|=n} f_0(S)=\sum_{S\subseteq [b-1], |S|=n-1} f_0(S),$$ and similarly for $f_c$.
	
	As we had $b>2n+c$, we have $b-1\geq 2n+c$ and $b-1\geq 2(n-1)+c$. Hence by induction, $\sum_{S\subseteq [b-1], |S|=n} f_0(S)=\sum_{S\subseteq [b-1], |S|=n} f_c(S)$ and $\sum_{S\subseteq [b-1], |S|=n-1} f_0(S)=\sum_{S\subseteq [b-1], |S|=n-1} f_c(S)$, which means that we are ready.
\end{proof}

\begin{remark} 
A cactus graph is a graph whose 2-connected components are all cycles or edges.
	Ohsugi and Tsuchiya shows \cite[Proposition 5.2]{OT} that if the 2-connected components of a graph $\cG$ are $\cG_1, \dots \cG_k$ then $h^*_\cG$ is the product of $h^*_{\cG_1}, \dots , h^*_{\cG_k}$. Hence one can also compute the $h^*$ polynomial of cactus graphs. We note that this can also be seen directly from Theorem \ref{thm:shelling_of_symm_edge_poly} as the Jaeger trees of different 2-connected components behave independently (and for a fixed ribbon structure and base point, the tour of any tree reaches a given 2-connected component at the same node-edge pair).
\end{remark}

\subsection{Complete bipartite graphs}

The $h^*$-polynomial of the symmetric edge polytope of complete bipartite graphs was determined by Higashitani, Jochemko and Micha\l{}ek \cite{arithm_symedgepoly} using Gröbner basis techniques. The computation of \cite{arithm_symedgepoly} can be divided into two parts. 
In the first part, they reduce the computation of the $h^*$-polynomial to a graph-theoretic counting problem. Then, in the second part, they solve this counting problem.
The first part of the problem can also be done using Jaeger trees, as we now explain.

In \cite{arithm_symedgepoly}, they construct a triangulation of $P_{K_{n,m}}$ using Gröbner bases. Then they interpret the simplices of the triangulation as certain oriented spanning trees of $K_{n,m}$. 
Then they prove that $h^*_i$ is equal to the number of trees in their triangulation that (for some fixed base node) have $i$ head-edges.

In \cite{semibalanced}, for each facet of $P_{K_{n,m}}$, we compute a dissection by Jaeger trees (which is actually a triangulation in that case). We use different ribbon structures for different facet graphs, but the base point can be chosen the same (by \cite[Remark 9.4]{semibalanced}).
The Jaeger trees have a simple geometric description. Moreover, they agree with the trees of \cite{arithm_symedgepoly}. Now the interpretation of $h^*_i$ as the number of Jaeger trees with $i$ head-edges is a direct application of Corollary \ref{cor:h^*_of_sym_edge_poly} (up to symmetry).

Even though it is slightly more involved to give a description of the trees of the triangulation based on the Jaeger language, we feel that the rest of the computation is more automatic, with the notion of Jaeger trees and the graph theoretic description of the $h^*$-polynomial (Corollary \ref{cor:h^*_of_sym_edge_poly}) at hand.

\subsection{Planar bipartite graphs}

Let $\cG$ be a planar bipartite graph. In this case, we can translate quantities about $P_\cG$ to quantities about the planar dual of $\cG$. 

We know that the facets of $P_\cG$ correspond to semi-balanced orientations of $\cG$. 
A planar directed graph is semi-balanced if and only if its (directed) planar dual is Eulerian; indeed, a cycle of a digraph $G$ is semi-balanced if and only if the corresponding cut in $G^*$ has equal number of edges going in the two directions.
Hence the number of facets of $P_\cG$ equals to the number of Eulerian orientations of $\cG^*$.

By \cite{semibalanced}, the number of Jaeger trees in a planar semi-balanced digraph $G$ is equal to the number of spanning arborescences of $G^*$ rooted at $r^*$ (for an arbitrary fixed choice of $r^*$).
Hence the number of Jaeger trees of $\cG$ is the sum of the numbers of spanning arborescences of Eulerian orientations of $\cG^*$.

\section{A geometric formula for the volume of $P_\cG$ for bipartite $\cG$}
\label{sec:geometric_formula_for_volume}

Finally, for bipartite graphs $\cG$ let us give a simple geometric formula for the volume of $P_\cG$.

Let $\cG$ be a bipartite graph. Then the facets of $P_\cG$ correspond to all the semi-balanced orientations of $\cG$. (That is, there are no hidden edges in any of them.)
Fix a ribbon structure and basis for $\cG$. This induces a ribbon structure and a basis for each facet graph $G_l$. As explained in Section \ref{sec:prep}, the volume of $P_\cG$ is the total numbers of Jaeger trees in the facets. 

Take an (unoriented) spanning tree $\mathcal{T}$ of the (undirected) graph $\cG$.
It is enough to tell for each spanning tree $\mathcal{T}$ of $\cG$ how many facet graphs $G_l$ are there such that (the appropriatelly oriented version of) $\mathcal{T}$ is a Jaeger tree in $G_l$ (for the fixed ribbon structure and basis).

We claim the following: for each tree $\mathcal{T}$ there is a point $p_T$ such that the facet graphs $G_l$ where $\mathcal{T}$ is a Jaeger tree correspond exactly to the facets of the symmetric edge polytope that are visible from $p_T$. 

The tour of $\mathcal{T}$ is independent of the orientation of the edges. Hence in order for $\mathcal{T}$ to be a Jaeger tree in $G_l$, each edge $e\notin \mathcal{T}$ has to be oriented so that it is cut at its tail in the tour of $\mathcal{T}$. This orientation only depends on $T$.

For each edge $e\notin T$, take the above mentioned orientation $\overrightarrow{e}$. Let $\mathbf{x}_{\overrightarrow{e}}$ be the vertex of the symmetric edge polytope corresponding to $\overrightarrow{e}$.

Let $p_T=(1+\delta)\frac{1}{|E - T|}\sum_{e\notin \mathcal{T}}\mathbf {x}_{\overrightarrow{e}}$,
where $0 < \delta < \frac{1}{|E|}$.

We know that each facet has a linear functional $l$ such that points of the facet have $l(\mathbf{p})=1$, while points of the polytope have $l(\mathbf{p})\leq 1$. This facet is visible from a point $\mathbf{p}$ if and only if $l(\mathbf{p})\geq 1$. 
It is enough to show that if $l$ is the fuctional defining the facet $G_l$ where $T$ is a Jaeger tree, then $l(p_T)>1$ and if $l$ is the fuctional defining the facet $G_l$ where $T$ is not a Jaeger tree, then $l(p_T) < 1$.

This is clear, since $T$ is a Jaeger tree exactly in facet graphs $G_l$ where each $e\notin \mathcal{T}$ is oriented as above.  
In these faces, $l(p_T)=1+ > 1$. If in a face, at least one edge of $E-T$ is oriented oppositely then for that edge $\overleftarrow{e}$, we have $l(\mathbf{x}_{\overrightarrow{e}})= -1$, hence $l(p_T)\leq (1+\delta)(1-\frac{2}{|E-T|}) < 1$.







\bibliographystyle{plain}
\bibliography{Bernardi}

\end{document}